\documentclass[10pt,a4paper]{article}
\usepackage[margin=1in]{geometry}
\usepackage{graphicx}
\usepackage[colorlinks=true, allcolors=black]{hyperref}
\usepackage{amsmath,amsthm,amsfonts,amssymb,amscd,mathrsfs}
\usepackage{enumerate}
\usepackage{mathrsfs}
\usepackage{xcolor}
\usepackage{graphicx}
\usepackage{hyperref}
\usepackage{enumitem}
\usepackage{bbm}
\usepackage{float}
\usepackage{fancyvrb}
\usepackage{booktabs}
\usepackage{subcaption}
\usepackage{algpseudocode,algorithmicx,algorithm}
\usepackage{etoolbox}
\patchcmd{\abstract}{\small}{}{}{}
\usepackage{arydshln}

\def\algbackskip{\hskip-\ALG@thistlm}
\newtheorem{theorem}{Theorem}[section]
\newtheorem{definition}[theorem]{Definition}

\newtheorem{lemma}[theorem]{Lemma}

\newtheorem{assumption}[theorem]{Assumption}
\newtheorem{problem}[theorem]{Problem}
\newtheorem{example}[theorem]{Example}
\newtheorem*{theorem*}{Problem}
\newtheorem{remark}[theorem]{Remark}

\usepackage{mathrsfs}
\usepackage{mathtools}

\allowdisplaybreaks[4]
\title{An Efficient Data-Driven Framework for Linear Quadratic Output Feedback Control\footnote{Corresponding author: Yuan-Hua Ni, and Yiqin Yang}
}
\date{\today}
\author{Jun Xie\footnote{College of Artificial Intelligence, Nankai University, China. Email: xiejun@mail.nankai.edu.cn}~~~~  Yuan-Hua Ni\footnote{College of Artificial Intelligence, and Tianjin Key Laboratory of Interventional Brain-Computer Interface and Intelligent Rehabilitation, Nankai University, China. Email: yhni@nankai.edu.cn}~~~~Yiqin Yang\footnote{The Key Laboratory of Cognition and Decision Intelligence for Complex Systems, Institute of Automation, Chinese Academy of Sciences, China. Email: yiqin.yang@ia.ac.cn}~~~~Bo Xu\footnote{The Key Laboratory of Cognition and Decision Intelligence for Complex Systems, Institute of Automation, Chinese Academy of Sciences, China. Email: xubo@ia.ac.cn}}
\begin{document}

\maketitle

\begin{abstract}
This paper focuses on linear quadratic regulator problems with unmeasurable states and unknown system parameters, aiming to achieve controller optimality while ensuring mild feasibility conditions and high data efficiency of the proposed learning algorithms.
To this end, this study first establishes a state parameterization framework to construct an effective substitute for unmeasurable states, centering on two full-row-rank properties: the mapping matrix from substitute states to true states, and the data matrix formed by substitute states. The former guarantees performance equivalence between output feedback and state feedback, while the latter ensures algorithm feasibility for both single-output and multi-output systems. To guarantee both properties, a projection-based state parameterization method is proposed. 
The proposed framework can also serve as an independent module integrated with other model-free algorithms or applied to other linear optimal control problems. 
To solve for the optimal output feedback controller, this study uses full-row-rank substitute state data matrices for data augmentation and proposes efficient model-free off-policy policy iteration and value iteration algorithms with theoretical guarantees on optimality, convergence, and robustness.
Compared with existing studies, the proposed algorithms are easier to implement, do not rely on least-squares regression, are strictly feasible for multi-output scenarios, and exhibit advantages in data requirements, unknown handling, and computational complexity. 
   
\end{abstract}

%%%%%%%%%%%%%%%%%%%%%%%%%%%%%%%%%%%%%%%%%%%
\section{Introduction}\label{section1}
Linear quadratic regulator (LQR) is a foundational optimal control method that minimizes a quadratic performance index to achieve stable and optimal control of linear systems with measurable states. It plays a critical role in enabling precise and efficient process control in key domains such as robotic control \cite{Robot}, aircraft control \cite{Deng-CT-output}, and autonomous driving \cite{Autonomous}. 
However, LQR faces two practically relevant yet inherently challenging realities: (i) the system matrices are seldom known a priori; (ii) system states are often unmeasurable due to cost-effectiveness and engineering practicality constraints \cite{Chen-LQG}. To address these challenges, this paper develops model-free learning algorithms to compute the optimal controller for the discrete-time LQR problem using only input-output data.

For LQR problems with unmeasurable states and unknown parameters, system identification can serve as a feasible approach, yet it entails complicated operations, high computational overhead, and typically yields suboptimal controllers \cite{Ivan-data-driven}. Accordingly, current research on model-free output feedback LQR predominantly concentrates on reinforcement learning algorithms. 
Typical approaches employ policy optimization algorithms to derive optimal static output feedback (SOF) and dynamic output feedback (DOF) controllers, but suffer from notable limitations. First, unmeasurable states lead to disconnected feasible regions and complex optimization structures, making algorithm convergence difficult to guarantee even in model-based settings \cite{Sof,dLQR}. Second, the optimal SOF or DOF controllers may still perform suboptimally for LQR problems \cite{Deng-DT-output}. Moreover, independently designing state observers for DOF controllers is extremely challenging in model-free settings. 
Another research line (\cite{DT-output,Deng-DT-output}) constructs high-dimensional input-output-related vectors as substitutes for unmeasurable states, thereby establishing equivalence to state feedback. Since the LQR optimal solution is an optimal state feedback controller, this method ensures convergence and optimality for the corresponding policy iteration (PI) and value iteration (VI) algorithms. However, this approach may encounter difficulties
in multi-output (MO) linear systems: redundant coordinates in the high-dimensional vectors render the least squares (LS) regression matrix rank-deficient, making the solution to iterative equations non-unique. 
Meanwhile, existing model-free algorithms for unmeasurable states still leave considerable room for improvement in computational efficiency. Model-free methods for measurable states commonly employ full-row-rank data matrices based on Willems’ Fundamental Lemma (\cite{Willems,De-Formulas}) to fully utilize system structures and enhance efficiency \cite{Efficient-Q,Deepo}. However, such advantages remain underexplored for problems with unmeasurable states.

To guarantee optimality, ensure feasibility in MO scenarios, and fully exploit data matrices,
this study first presents a more comprehensive state parameterization framework than those in \cite{DT-output,Deng-DT-output,State-Para}.
Under the controllability assumption, the framework shows that the substitute state defined in Definition \ref{def2} always exists, which serves as an equivalent representation of the unmeasurable state and helps ensure algorithm optimality. 
In particular, the framework focuses on the row-rank property of the data matrices formed by the substitute states. By analyzing the consistent system (Definition \ref{def1}) and the dimension of the minimal linear subspace containing all possible substitute states, we prove that the commonly used substitute state data matrix can be rank-deficient in MO scenarios.
To remedy this, the framework constructs an improved substitute state via a projection, which reduces redundant information, guarantees full row rank of the data matrix, and enables controller parameterization using the data matrix.
Building on this framework, this study proposes data-efficient model-free off-policy PI and VI algorithms with comprehensive theoretical guarantees. The iterative equations are enhanced by the full-row-rank improved substitute state data matrices, eliminating the need to convert the iterative equations into LS problems for a solution and guaranteeing algorithm convergence. This avoids additional data requirements for regression matrices, which are difficult to satisfy in MO settings (as shown in Subsection \ref{section4_3}).

\subsection{Related Works}\label{section1_1}		
\textbf{(i) Willems' Fundamental Lemma.}	
Willems' Fundamental Lemma \cite{Willems} demonstrates that for a controllable discrete-time linear time-invariant (DT-LTI) system, a persistently exciting (PE) input-state Hankel matrix has full row rank, while a PE input-output Hankel matrix can be used to generate all valid system trajectories. This lemma has found extensive applications in autonomous driving \cite{DeeP}, power grids, and quadrotor control \cite{Ivan-data-driven}.
For single-output (SO) systems, historical input-output trajectories can directly replace unknown states without losing the full-row-rank property of Hankel matrix \cite{De-Formulas}. However, this conclusion cannot be readily extended to MO systems. To address MO problems, the work in \cite{Notes-MO} constructs a full-row-rank data matrix by selecting linearly independent rows from the historical input-output trajectory matrix. 
This motivates further study of the rank of data matrices and dimension reduction for more general substitute state constructions.

\textbf{(ii) Model-Free State Feedback LQR.} 
Model-free state feedback results serve as foundational and inspirational references for model-free output feedback LQR problems. 
In contrast to linear matrix inequalities methods \cite{Low-SDP}, more ``brute-force'' reinforcement learning approaches may offer better computational efficiency. The work in \cite{LQR-PO} reformulates LQR as an optimization problem and establishes the gradient domination condition, a key criterion for the global convergence of policy gradient. 
Studies in \cite{PI-noises} and \cite{VI-noises} develop model-free PI and VI algorithms for stochastic state feedback LQR, respectively. Furthermore, the research in \cite{Efficient-Q} integrates Willems' Fundamental Lemma to propose an off-policy Q-learning algorithm for state feedback LQR, characterized by fewer unknowns and higher computational efficiency.

\textbf{(iii) Output Feedback LQR.}
Owing to the complex topological structure of output feedback LQR \cite{Topo-MO}, gradient domination conditions are hard to satisfy for both SOF and DOF, resulting in a lack of global convergence guarantees for policy gradient \cite{dLQR,Sof}. 
An alternative approach is to reconstruct or estimate states, which can also be applied to model-free problems. For instance, the research in \cite{Lewis-IOH} uses historical input-output trajectories as a substitute for states in PI and VI; the study in \cite{DT-output,CT-output} employs substitute states generated by a user-defined internal model to replace states in off-policy PI and VI, though the iteration process may be affected by observation errors. This issue is addressed in \cite{Deng-DT-output,Deng-CT-output} by incorporating observation error replication into the internal model. However, the aforementioned value-based iterative methods neither eliminate redundant substitute state information nor avoid explicit LS regression, making their feasibility conditions difficult to satisfy for MO systems, as shown in Subsection \ref{section4_3}. 
Additionally, this state reconstruction strategy is also applicable to linear quadratic Gaussian problems. The works in \cite{3A_CDC,3A_CSL} derive the closed-form data-driven solution for such problems by employing data matrices constructed from historical input-output trajectories.

\subsection{Comparisons and Contributions}\label{section1_2}
Compared with existing results, the main contributions of this paper are as follows.

\begin{itemize}

	\item[(i)] A general framework for learning optimal output feedback LQR controllers is established using only input-output data. This framework exhibits high applicability and scalability, stemming from three key advantages. First, it only requires a basic PE condition, with no additional restrictive data assumptions; second, it is applicable to both SO and MO problems; third, it can potentially be extended to other linear optimal control problems such as continuous-time settings.

    This framework obviates the need for multiple attempts to satisfy the full-row-rank property of the regression matrix in \cite{DT-output,Deng-DT-output} and is guaranteed to solve MO problems optimally. 
    In contrast to \cite{3A_CSL}, this framework does not focus on the design of observer gains or their substitutes, and provides a systematic method to handle the rank deficiency of substitute state data matrices.  It also addresses the challenge of extending the state feedback results in \cite{Efficient-Q} and \cite{VI-noises} (the latter without data augmentation) to output feedback with data augmentation.

	\item[(ii)] A generalized state parameterization method is proposed with control-theoretic support. We analyze the row rank of data matrices for two typical substitute states, revealing that rank-deficiency stems from the mismatch between the dimension of the minimal linear subspace containing all possible substitute states and the nominal dimension. 
    By projection, we construct an improved substitute state with lower redundancy, enabling full-row-rank data matrices. This simplifies the design and computation of iterative algorithms, extends the full row rank conclusion of Willems' Fundamental Lemma to MO scenarios, and can be embedded as a standalone module into other data-driven algorithms such as policy optimization.

    %This state parameterization method provides theoretical grounding for the physical intuition in \cite{Notes-MO}. Notably, the findings in \cite{Notes-MO} and the vector autoregressive framework with exogenous input \cite{Varx} only consider the scenario in Subsection \ref{section3_1}, and can be regarded as special cases of the proposed method.

    Notably, the state parameterization results in \cite{Notes-MO} for systems without direct feedthrough and the vector autoregressive framework with exogenous input \cite{Varx} are related to the historical trajectory-based state parameterization in Subsection \ref{section3_1}, whereas our framework additionally accommodates observer-based state parameterizations.

	\item[(iii)] Data-efficient model-free off-policy PI and VI algorithms are proposed for learning the optimal output feedback LQR controllers,
    breaking the traditional paradigm of reliance on explicit LS regression. Moreover, the corresponding proofs for stability, convergence, optimality, and robustness are established.

	Compared with \cite{Deng-DT-output}, in terms of data utilization, the proposed algorithms relax the rank conditions on data matrices required for convergence in \cite{Deng-DT-output}, and require a smaller amount of data; in terms of computational efficiency, the proposed algorithms involve fewer unknown parameters and have lower theoretical computational complexity.	
\end{itemize}

The rest of the paper is organized as follows. Section \ref{section2} introduces the problem formulation and related preliminaries. Section \ref{section3} presents a state parameterization framework for constructing effective substitutes for unmeasurable states.
Section \ref{section4} develops efficient model-free PI and VI algorithms applicable to MO problems. Section \ref{section5} provides the robustness analysis and further detailed discussions. Section \ref{section6} demonstrates the feasibility and advantages of the proposed algorithms through numerical experiments. Finally, the paper ends with the conclusion in Section \ref{section7}.

\textbf{Notations.}
$\mathbb{R}$, $\mathbb{C}$, $\mathbb{S}^n$ and $\mathbb{S}^n_{++}$ ($\mathbb{S}^n_+$) denote the set of real numbers, the set of complex numbers, the set of $n\times n$ symmetric matrices, and the set of $n\times n$ positive (semi-)definite matrices, respectively.
$A^{\frac{1}{2}}$ denotes the unique positive semi-definite square root of $A\in\mathbb{S}_{+}^{n}$. 
$I_n$ is an $n\times n$ identity matrix, and $\mathbf{0}$ denotes a zero vector or matrix of appropriate dimension. Let $\Vert\cdot\Vert$ denote the Euclidean vector norm or spectral matrix norm. 
For a matrix $A$, 
$\mathrm{rank}(A)$, $\rho(A)$, $A^\top$, $A^\dagger$, $\sigma_{\min}(A)$, $\sigma_{\max}(A)$, $\mathrm{im}(A)$, and $\mathrm{ker}(A)$ denote its rank, spectral radius, (conjugate) transpose, Moore-Penrose pseudo-inverse, minimum singular value, maximum singular value, image space, and kernel space, respectively.
For a square matrix $B$, $\mathrm{det}(B)$ and $\mathrm{adj}(B)$ denote its determinant and adjugate matrix, respectively. 
Define $\mathrm{diag}(A,B)$ as the block diagonal matrix with main diagonal blocks $A$ and $B$.
Let $z$, $\mathcal{Z}(\cdot)$ and $\mathcal{Z}^{-1}(\cdot)$ denote the complex variable in the z-domain, the z-transform and the inverse z-transform, respectively. For $A=[a_1,\cdots,a_m]\in\mathbb{R}^{n\times m}$ and $B=(b_{ij})\in\mathbb{S}^n$, define $\mathrm{vec}(A):=[a_1^\top,\cdots, a_m^\top]^\top\in\mathbb{R}^{mn}$ and $\mathrm{vech}(B):=[b_{11},$
$\sqrt{2}b_{12},\cdots,\sqrt{2}b_{1n},b_{22},\sqrt{2}b_{23},\cdots,\sqrt{2}b_{n-1,n},b_{nn}]^\top$.
The notation $\otimes$ is defined as the Kronecker product.
For a signal sequence $\{u_t\}$, define its stacked window vector as $u_{[i,j]}:=[u_i^\top,u_{i+1}^\top,\cdots,u_{j}^\top]^\top$, where $i$ and $j$ are integers with $i<j$. Accordingly, a Hankel matrix of depth $N$ is defined as $\mathcal{H}_{N}(u_{[0,T-1]}):=[u_{[0,N-1]},u_{[1,N]},\cdots,u_{[T-N,T-1]}]$. A triple $(\mathcal{A},\mathcal{B},\mathcal{C})$ represents a DT-LTI system.
Define
$\mathcal{R}_{N}(A,B):=[A^{N-1}B,\cdots,AB,B]$, $\mathcal{O}_N(A,C):=[C^\top,(CA)^\top,\cdots,(CA^{N-1})^\top]^\top$, and
$$\mathcal{T}_N(A,B,C):=\begin{bmatrix}
	\mathbf{0}&&&\\
	CB&\mathbf{0}&&\\
	\vdots&\ddots&\ddots&\\
	CA^{N-2}B&\cdots&CB&\mathbf{0}
\end{bmatrix};$$
 the subscript $N$ will be omitted when the context is clear.

%%%%%%%%%%%%%%%%%%%%%%%%%%%%%%%%%%%%%%%%%%%%%%%%%%%%%%%%

\section{Problem Formulation and Preliminaries}\label{section2}
This section introduces the problem formulation and the outcomes of direct data-driven control.

\subsection{Problem Formulation}\label{section2_1}
Consider a DT-LTI system of the form
\begin{equation}\label{system}
	\begin{aligned}
		x_{t+1}&=Ax_{t}+Bu_{t},\quad y_{t}=Cx_{t}
	\end{aligned}
\end{equation}
with unmeasurable state $x_{t}\in\mathbb{R}^n$, measurable input $u_{t}\in\mathbb{R}^m$ and output $y_{t}\in\mathbb{R}^p$, where $p\leq n$ and the output matrix $C$ may be rank-deficient.
It is necessary to define an admissible controller set $\mathcal{U}_{ad}$ based on the adopted controller structure to guarantee the stability of the target controller.
Define an infinite horizon quadratic cost functional as
\begin{equation}\label{J}
	\begin{aligned}
		J(x_{0}, u)
        =\sum_{t=0}^{\infty}\left(y_{t}^\top Qy_{t}+u_{t}^\top Ru_{t}\right)
	\end{aligned}
\end{equation}
with constant cost weighting matrices $Q\in\mathbb{S}^{p}_{+}$ and $R\in\mathbb{S}^{m}_{++}$.
Additionally, $Q_x:=C^\top QC\in\mathbb{S}^{n}_{+}$ acts as the state cost weighting matrix.  
This paper aims to address the model-free output feedback LQR problem, which is described as follows.
\begin{problem}\label{P1}
Consider system (\ref{system}) with unknown matrices $A$, $B$, $C$ and unmeasurable states $\{x_t\}$. Using only measurable input-output data $\{u_t,y_t\}$, find the optimal stabilizing control policy $u^*$ that minimizes the cost functional (\ref{J}), i.e.,
	\begin{equation}\label{pmin}
		u^*=\arg\min_{u\in\mathcal{U}_{ad}} J(x_{0},u).
	\end{equation}
\end{problem}

We introduce the following standard assumption throughout this paper.
\begin{assumption}[Controllability and Observability]\label{as1}
	The pair $(A, B)$ is controllable, and the pairs $(A, C)$, $(A, Q_x^{\frac{1}{2}})$ are observable.
\end{assumption}

\begin{remark}
In general, the stabilizability of $(A,B)$ suffices for the existence of an optimal feedback controller \cite{LinearSys}. Since this paper employs the data parameterization method based on Willems' Fundamental Lemma, we assume that $(A,B)$ is controllable here; this is consistent with key literature in this field \cite{Ivan-Behavioral,De-Formulas,Efficient-Q}. The relaxation of this assumption will be further discussed in subsequent sections.
	The observability of $(A,C)$ guarantees that the unknown states can be uniquely reconstructed from the input-output data of system (\ref{system}) \cite{Deng-CT-output}. From Lyapunov control theory, the observability of $(A,Q_x^{\frac{1}{2}})$ is a prerequisite for the convergence of iteration-based schemes \cite{con-AQ}. 			
\end{remark}

When $A$, $B$ are known and the states $\{x_t\}$ are measurable (i.e., $C=I_n$), Problem \ref{P1} reduces to finding the optimal state feedback controller of the form $u_t=K_xx_t$. In this case, the cost functional (\ref{J}) can be rewritten as the quadratic form $J(x_0,K_x)=x_0^\top P^{K_x}x_0$, where $P^{K_x}\in\mathbb{S}^{n}_{++}$ is the unique solution to the following Lyapunov equation:
\begin{equation}\label{lyap}
	P^{K_x}=(A+BK_x)^{\top}P^{K_x}(A+BK_x)+Q_x+K_x^\top RK_x.
\end{equation}
The optimal state feedback gain is given by $K_x^*=-(B^\top P^{K_x^*}B+R)^{-1}B^\top P^{K_x^*}A$, where $P^{K_x^*}\in\mathbb{S}^{n}_{++}$ is the unique solution to the following discrete algebraic Riccati equation (DARE): 
\begin{equation}\label{DARE}
 \begin{split}
     P^{K_x^*}=&Q_x + A^{\top}P^{K_x^*}A\\
 &- A^{\top}P^{K_x^*}B(R + B^{\top}P^{K_x^*}B)^{-1}B^{\top}P^{K_x^*}A.
 \end{split}
\end{equation}

When $A$, $B$, $C$ are known but the states $\{x_t\}$ are unmeasurable (especially when $p< n$), restricting the feedback control structure in Problem \ref{P1} to either the SOF controller $u_t=K_yy_t$ or
the conventional DOF controller 
\begin{equation}\label{hatK} 	
\hat{K}: \quad x^c_{t+1}=A_Kx^c_{t}+B_Ky_t,\quad u_t=C_Kx^c_{t}
\end{equation}
entails two critical issues. First, one must further address the existence of the optimal SOF controller \cite{Guass-output} or the impact of similarity transformations on the cost of the conventional DOF controller \cite{dLQR}. More importantly, as shown in Subsection 4.3 of \cite{Deng-DT-output}, the optimal solutions obtained from SOF or DOF structures are not necessarily optimal for Problem \ref{P1}. 
Motivated by these observations, this paper will adopt a modified output feedback structure $u_t=Fv_t$ (with $v_t$ defined in Lemma \ref{le4}) which enjoys performance equivalence with state feedback \cite{State-Para}. This structure yields stronger optimality guarantees for Problem \ref{P1} and enables a simpler algorithm design.

\subsection{Willems' Fundamental Lemma}\label{section2_3}
This subsection presents the results of direct data-driven control. Considering the state, input, and subsequent state sequences of length $T$,
\begin{equation}\label{data_mar}
	\begin{aligned}
		X_0&:=[x_0, x_1, \cdots, x_{T-1}]\in\mathbb{R}^{n\times T},\\
		U_0&:=[u_0, u_1, \cdots, u_{T-1}]\in\mathbb{R}^{m\times T},\\
		X_1&:=[x_1, x_2, \cdots, x_{T}]\in\mathbb{R}^{n\times T},
	\end{aligned}
\end{equation}
system (\ref{system}) can be written in data matrix form as $X_1=AX_0+BU_0$. The Hankel matrix constructed from PE input-state data enables non-parametric full characterization of system (\ref{system}).

\begin{definition}[Persistently Exciting (PE)]\label{PE}
A signal sequence $u_{[0,T]}$ is PE of order $N$, if the Hankel matrix $\mathcal{H}_{N}(u_{[0,T]})$ has full row rank, where $T\geq (m+1)N-1$.
\end{definition}

\begin{lemma}[Willems' Fundamental Lemma \cite{Willems}]\label{willems}
	Assume the DT-LTI system (\ref{system}) is controllable. If the input sequence $u_{[0,T-1]}$ of system (\ref{system}) is a PE signal of order $n+N$, with the corresponding state sequence $x_{[0,T]}$ and output sequence $y_{[0,T-1]}$, then the following assertions hold. 
	\begin{itemize}
		\item[\textup{(a)}] 	$\mathrm{rank}\left(\begin{bmatrix}
			\mathcal{H}_{N}(u_{[0,T-1]})\\\mathcal{H}_{1}(x_{[0,T-N]})
		\end{bmatrix}\right)=mN+n$.
		\item[\textup{(b)}] 
		Assume $N\geq \ell$ with $\ell$ being the observability index of system (\ref{system}). The sequence $(\bar{u}_{[0,N-1]},\bar{y}_{[0,N-1]})$ is an input-output trajectory of length $N$ for system (\ref{system}) if and only if there exists a vector $\alpha\in\mathbb{R}^{T-N\textcolor{blue}{+1}}$ such that
		\begin{equation*}\label{alp}
			\begin{bmatrix}
				\mathcal{H}_N(u_{[0,T-1]})\\\mathcal{H}_N(y_{[0,T-1]})
			\end{bmatrix}\alpha=\begin{bmatrix}
				\bar{u}_{[0,N-1]}\\\bar{y}_{[0,N-1]}
			\end{bmatrix}.
		\end{equation*}
	\end{itemize}
\end{lemma}
When $N=1$, Hankel matrices degenerate into the form (\ref{data_mar}), and according to \cite{New-Per}, the corollary of assertion (a) in Lemma \ref{willems} (that $[X_0^\top, U_0^\top]^\top$ has full row rank) is
\begin{equation}\label{IK}
    \mathrm{im}([I_n,K_x^\top]^\top)\subset\mathrm{im}([X_0^\top, U_0^\top]^\top).
\end{equation}
This corollary guarantees the existence of a matrix $\Phi$ satisfying
$X_0\Phi=I_n$ such that $A+BK_x=X_1\Phi$ and $K_x=U_0\Phi$; this representation is often used to design the target controller in the model-free setting \cite{New-Per,Deepo}. Note that assertion (b) reveals that the input-output data matrix contains all the information characterizing system (\ref{system}). Thus, when the states $\{x_t\}$ are unmeasurable, a common intuition is to directly replace unknown $X_0$ with a data matrix constructed from input-output trajectories. However, the row full-rank property of the resulting data matrix and the existence of the parameterizing matrix $\Phi$ are not guaranteed, which creates difficulties for data parameterization and algorithm design, as demonstrated in Theorem \ref{thm3} and Subsection \ref{section4_3}.  Accordingly, Section \ref{section3} develops a remedy for these issues.

%%%%%%%%%%%%%%%%%%%%%%%%%%%%%%%%%%%%%%%%%%%%%%%%%%%

\section{State Parameterization}\label{section3}

This section aims to construct a substitute state $v_t$ for the unmeasurable true state $x_t$ using only input-output data, such that the corresponding data matrix $V_0:=[v_0,v_1,\cdots,v_{T-1}]$ has full row rank under the PE condition.

\begin{definition}[Substitute State]\label{def2}
    A signal $v_t$ is called a substitute state of the unmeasurable true state $x_t$ if there exist matrices $\hat{\mathcal{A}}$, $\hat{\mathcal{B}}_u$, $\hat{\mathcal{B}}_y$, a full-row-rank matrix $F$, and a prescribed initialization rule such that: (i) $v_0$ is constructed from the available input-output history or prescribed auxiliary initial conditions, subject to the compatibility condition $x_0=Fv_0$, and subsequent $v_t$ are updated via the implementable linear recursion $v_{t+1}=\hat{\mathcal{A}}v_t+\hat{\mathcal{B}}_uu_t+\hat{\mathcal{B}}_yy_t$, where $u_t$, $y_t$ are the input and output of system (\ref{system}); (ii) under this compatible initial setup, $x_t = F v_t$ holds for all $t\geq0$ and any input sequence of system (\ref{system}). Consequently, $F$ is called a state parameterization matrix. 
\end{definition}

We now explain the motivation for constructing such a substitute state. Under $x_t=Fv_t$, it holds that $u_t=K_xx_t=K_xFv_t:=Kv_t$. Under compatible initialization, using $u_t=K_xx_t$ or $u_t=Kv_t$ on system (\ref{system}) produces identical input-output trajectories and the same cost; that is, the two control laws are performance-equivalent. Moreover, according to \cite{State-Para}, the corresponding quadratic value functions are related by $P=F^\top P^{K_x} F$, which provides a basis for connecting the associated value equations and policy updates. 
In particular, $K^*=K_x^*F$ achieves the same optimal cost as the optimal state-feedback gain $K_x^*$. These relationships underpin the subsequent optimality and convergence analysis in Section \ref{section4}.

Next, we explain why the data matrix $V_0$ is required to have full row rank. Since $v_t$ is used to replace the unmeasurable true state $x_t$, we expect $V_0$ to inherit key properties from $X_0$, that is, full-row-rank (analogous to assertion (a) in Lemma \ref{willems}). 
This property ensures that $V_0$ can fully characterize the feedback gain $K$ (see Theorem \ref{thm3}, analogous to the corollary in equation (\ref{IK})), and enables the design of efficient model-free algorithms that fully exploit the structure of system (\ref{system}) and the associated data matrices (see Section \ref{section4}).  Conversely, if $V_0$ does not have full row rank, this may lead to non-convergence, numerical instability, and other issues in subsequent algorithm design.

Subsections \ref{section3_1} and \ref{section3_2} respectively analyze the row-rank properties of the state parameterization matrices and the corresponding data matrices for two typical substitute states. In contrast to existing work, we make two key observations. First, we show that the controllability of system (\ref{system}) guarantees that the state parameterization matrices have full row rank. Second, from a direct data-driven perspective, 
we prove that in MO settings, the corresponding substitute state data matrices may never have full row rank even under PE inputs. Building on these results, Subsection \ref{section3_3} constructs an improved substitute state via projection, and the resulting data matrix has full row rank in both SO and MO settings, thus retaining all the aforementioned desirable properties.

\begin{definition}[Systems Consistency]\label{def1}
A DT-LTI system $v_{t+1}=\mathcal{A}v_t+\mathcal{B}u_t, y_t=\mathcal{C}v_t$ is said to be consistent with system (\ref{system}) if the following conditions hold: (i) $v_t$ is a substitute state as defined in Definition \ref{def2}; (ii) for any input sequence $\{u_t\}$, the two systems produce the same output sequence $\{y_t\}$ under compatible initial conditions $x_0=Fv_0$.
\end{definition}

\subsection{Historical Trajectory-Based State Parameterization}\label{section3_1}	
This subsection shows that the finite-length historical input-output trajectory $\xi_{t}:=[u_{[t-N,t-1]}^\top,y_{[t-N,t-1]}^\top]^\top$\footnote{Note that $\xi_t$ depends on the window length $N$; we suppress this dependence in the notation for brevity.} proposed in \cite{Lewis-IOH} can serve as a substitute state for unmeasurable $x_t$, and analyzes the row rank properties of its corresponding data matrix $\Xi_0:=[\xi_0,\xi_1,\cdots,\xi_{T-1}]$.

\begin{theorem}
\label{thm1}
	Consider system (\ref{system}). If Assumption \ref{as1} holds and $N\geq \ell$ with $\ell$ being the observability index of system (\ref{system}), then the following assertions hold.
	\begin{itemize}
		\item[\textup{(a)}] The state $x_t$ of system (\ref{system}) can be expressed in terms of the historical input-output trajectory $\xi_t\in\mathbb{R}^{(m+p)N}$ as
		\begin{align}\label{x_t}
			 x_t=[\mathcal{R}(A,B)-A^N\mathcal{O}^\dagger(A,C)\mathcal{T}(A,B,C),A^N\mathcal{O}^\dagger(A,C)]\xi_t:=M\xi_t.
		\end{align}
		\item[\textup{(b)}] The state parameterization matrix $M$ has full row rank.
	\end{itemize}
\end{theorem}
\begin{proof}
The proof of assertion (a) can be found in \cite{Lewis-IOH}. 
    For assertion (b), noting
	\begin{equation}\label{eq1}
		M\underbrace{\begin{bmatrix}
		I_{mN}&\mathbf{0}\\\mathcal{T}(A,B,C)&I_{pN}
		\end{bmatrix}}_{\mathcal{I}_M}=[\mathcal{R}(A,B),A^N\mathcal{O}^\dagger(A,C)],
	\end{equation}
    the full row rank of $M$ follows directly from the full row rank of $\mathcal{R}(A,B)$ (guaranteed by the controllability of system (\ref{system})) and  the invertibility of $\mathcal{I}_M$.
\end{proof}

\begin{remark}\label{re2}
The controllability of system (\ref{system}) is a sufficient condition for $M$ to have full row rank.
Suppose system (\ref{system}) is stabilizable but uncontrollable. If the dynamic matrix $A$ in system (\ref{system}) is nonsingular, $M$ still has full row rank according to equation (\ref{eq1}). If $A$ has zero eigenvalues, $M$ may not have full row rank; this is consistent with \cite{State-Para}.
\end{remark}

\begin{remark} 
Theorem \ref{thm1} requires $N\geq\ell$, but the observability index $\ell$ is difficult to compute in a model-free manner. Note that the state dimension $n$ of system (\ref{system}) satisfies $n\geq\ell$. Moreover, with sufficiently long, PE, noise-free data, $n$ can be identified model-freely by increasing a positive integer $k$ and computing
 $\mathrm{rank}\left([\mathcal{H}_{k}(u_{[0,T-1]})^\top,\mathcal{H}_{k}(y_{[0,T-1]})^\top]^\top\right)-mk$, 
 where $n$ is identified as the value at which this quantity levels off as $k$ increases \cite{Ivan-data-driven}. Therefore, we can empirically set $N=n$ without incurring excessive computational overhead from an overly large $N$.
\end{remark}

The following Lemma \ref{le1} reveals that, in MO settings (i.e., $p>1$), three phenomena may occur: the mapping from $\xi_t$ to $x_t$ is not necessarily unique; the consistent system (\ref{M_0}) with $\xi_t$ as the state variable is uncontrollable; and the substitute state data matrix $\Xi_0 := [\xi_{0}, \dots, \xi_{T-1}]$ is rank-deficient.

\begin{lemma}\label{le1}
	Define $\mathcal{A}_m:=\Omega+\Gamma C(M+M_0)$, $\mathcal{B}_m:=[\mathbf{0}_{m\times m(N-1)},I_{m},\mathbf{0}_{m\times pN}]^\top$, where
	$$\Omega:=\left[
	\begin{array}{c:c}
		\begin{array}{c:c}
			&I_{m(N-1)}\\\hdashline\mathbf{0}_{m\times m}&
		\end{array}&\mathbf{0}_{mN\times pN}\\\hdashline
		\mathbf{0}_{pN\times mN}&\begin{array}{c:c}
			&I_{p(N-1)}\\\hdashline\mathbf{0}_{p\times p}&
		\end{array}
	\end{array}\right],$$
    $\Gamma:=[\mathbf{0}_{p\times(mN+p(N-1))},I_{p}]^\top$,	and $M_0$ is any matrix satisfying $M_0\mathcal{M}_{\xi}=\mathbf{0}$,
    $$\mathcal{M}_{\xi}:=\begin{bmatrix}
			I_{mN}&\mathbf{0}\\\mathcal{T}(A,B,C)&\mathcal{O}(A,C)
		\end{bmatrix}.$$
	If Assumption \ref{as1} holds and $N\geq \ell$, the following assertions hold.
	\begin{itemize}
	  \item[\textup{(a)}] The relationship between the state $x_t$ of system (\ref{system}) and the historical input-output trajectory $\xi_{t}$ can be expressed by $x_t=(M+M_0)\xi_t$.
	  \item[\textup{(b)}] System 
	  \begin{equation}\label{M_0}
	  	\begin{aligned}
	\xi_{t+1}&=\mathcal{A}_m\xi_t+\mathcal{B}_mu_t,\quad
	  		y_t=C(M+M_0)\xi_t
	  	\end{aligned}
	  \end{equation}
  is consistent with system (\ref{system}).
     \item[\textup{(c)}]
	  The controllability matrix $\mathcal{R}(\mathcal{A}_m,\mathcal{B}_m)$ of system (\ref{M_0}) is of rank $mN+n$. System (\ref{M_0}) is controllable if and only if $pN=n$; otherwise, it is uncontrollable.
      \item[\textup{(d)}] Suppose the inputs of system (\ref{system}) satisfy the $(n+N)$-th order PE condition. Then the data matrix $\Xi_0:=[\xi_{0},\dots,\xi_{T-1}]$ is of rank $mN+n$, and it has full row rank if and only if $pN = n$.
	\end{itemize}
    
\end{lemma}
\begin{proof}		
Define
   $\mathcal{V}_{\xi}:=\{\xi_t=[u_{[t-N,t-1]}^\top,y_{[t-N,t-1]}^\top]^\top\vert \xi_t$ is any input-output trajectory of length $N$ of system (\ref{system})$\}$. Noting that for any $\xi_t\in\mathcal{V}_{\xi}$, it holds that
   \begin{equation}\label{TO}
       \xi_t=\mathcal{M}_{\xi}[u_{[t-N,t-1]}^\top,x_{t-N}^{\top}]^\top,
   \end{equation}
   where $\mathcal{M}_{\xi}$ has full column rank due to the observability of system (\ref{system}), i.e., $\mathrm{rank}(\mathcal{M}_{\xi})=mN+n$, and 
   $x_{t-N}$ can take any value in the reachable subspace owing to the controllability of system (\ref{system}). Hence, $\mathcal{V}_{\xi}= \mathrm{im}(\mathcal{M}_{\xi})$. We now prove the existence of $M_0$. Choose a matrix $\mathcal{O}_{\perp}\in\mathbb{R}^{pN\times(pN-n)}$ satisfying  $\mathcal{O}_{\perp}^\top \mathcal{O}_{\perp}=I_{pN-n}$ and $\mathrm{im}(\mathcal{O}_{\perp})=\mathrm{ker}(\mathcal{O}(A,C)^\top)$, and define $\mathcal{P}_N=\mathcal{O}_{\perp}^\top[-\mathcal{T}(A,B,C),I_{pN}]\in\mathbb{R}^{(pN-n)\times (m+p)N}$. Then for any $\xi:=[a^\top,b^\top]^\top$, $\mathcal{P}_N\xi=\mathbf{0}\Leftrightarrow b-\mathcal{T}(A,B,C)a\in\mathrm{im}(\mathcal{O}(A,C))\Leftrightarrow \xi\in\mathcal{V}_{\xi}$. Thus, $\mathrm{ker}(\mathcal{P}_N)=\mathcal{V}_{\xi}$. Let $M_0=\mathcal{N} \mathcal{P}_N$, where $\mathcal{N}\in\mathbb{R}^{n\times (pN-n)}$ is an arbitrary matrix. Consequently, $M_0\mathcal{M}_{\xi}=\mathbf{0}$. Conversely, if $M_0\mathcal{M}_{\xi}=\mathbf{0}$, then every row of $M_0$ lies in the orthogonal complement of $\mathcal{V}_{\xi}$, i.e., the span of the rows of $\mathcal{P}_N$, which implies that $M_0=\mathcal{N} \mathcal{P}_N$. Furthermore, combining with Theorem \ref{thm1}, assertions (a) and (b) follow readily.
   
   For assertion (c), we first prove that $\mathcal{V}_{\xi}=\mathrm{im}(\mathcal{R}(\mathcal{A}_m,\mathcal{B}_m))$. For any $\xi\in\mathrm{im}(\mathcal{R}(\mathcal{A}_m,\mathcal{B}_m))$, the above system consistency implies $\xi\in\mathcal{V}_{\xi}$. Hence, $\mathrm{im}(\mathcal{R}(\mathcal{A}_m,\mathcal{B}_m))\subseteq \mathcal{V}_{\xi}$. Conversely, take any $\xi=\mathcal{M}_{\xi}[a^\top,c^\top]^\top\in\mathcal{V}_{\xi}$. 
   By the controllability of $(A,B)$, there exists an input sequence that drives system (\ref{system}) from zero state to $c$ in $h$ steps. Accordingly, for system (\ref{M_0}) initialized at zero, we first apply this input sequence and then the input sequence $a$ of length $N$, which yields $\xi_{h+N}=\xi$.
   Thus, $\xi\in \mathrm{im}(\mathcal{R}(\mathcal{A}_m,\mathcal{B}_m))$, and consequently $\mathcal{V}_{\xi}\subseteq\mathrm{im}(\mathcal{R}(\mathcal{A}_m,\mathcal{B}_m))$. From the foregoing analysis, it follows that $\mathrm{im}(\mathcal{M}_\xi)=\mathrm{im}(\mathcal{R}(\mathcal{A}_m,\mathcal{B}_m))$. 
   Thus, $\mathrm{rank}(\mathcal{R}(\mathcal{A}_m,\mathcal{B}_m))=\mathrm{rank}(\mathcal{M}_{\xi})=mN+n$. 
   Note that a system is controllable if and only if the rank of its controllability matrix equals the dimension of the state. The state dimension of system (\ref{M_0}) is $(m+p)N\geq mN+n$, with equality if and only if $pN=n$. For assertion (d), by Lemma \ref{willems} and equation (\ref{TO}), $\mathrm{rank}(\Xi_0)=mN + n$ under $(n+N)$-th order PE inputs.
\end{proof}

For Lemma \ref{le1}, we make the following remarks. Since $M_0$ annihilates $\mathcal{V}_\xi$, we can choose $M_0=\mathbf{0}$ in the subsequent analysis. Assertion (c) can also be proved by applying Willems' Fundamental Lemma and a contradiction argument \cite{Notes-MO}. 
From the proofs of assertions (c) and (d), the rank deficiency of $\Xi_0$ in MO settings stems from a mismatch between the dimension of $\mathcal{V}_{\xi}$ and the nominal dimension $(m+p)N$ of $\xi_t$. Since $\mathcal{V}_{\xi}=\mathrm{im}(\mathcal{R}(\mathcal{A}_m,\mathcal{B}_m))$, this rank deficiency can also be attributed to the uncontrollability of the consistent system (\ref{M_0}).

\subsection{Observer-Based State Parameterization}\label{section3_2}
This subsection shows that the input-output filtered vector $\eta_t$ (defined in equation (\ref{x_t4})) can serve as a substitute state for the unmeasurable $x_t$. The vector is constructed based on a Luenberger observer and incorporates the replication of observation errors, as proposed in \cite{Deng-DT-output}. This subsection also analyzes the row rank properties of the corresponding data matrix $H_0:=[\eta_0,\eta_1,\cdots,\eta_{T-1}]$.

The estimated state $x_t^c\in\mathbb{R}^n$ obtained from the Luenberger observer satisfies
\begin{equation}\label{x^c}
	x^c_{t+1}=(A-LC)x^c_{t}+Bu_t+Ly_t,
\end{equation}
where $x^c_0=\mathbf{0}$ is typically chosen, and $L$ is a Luenberger observer gain. 
Define the observation error as $\epsilon_t:=x_t-x_t^c$, which satisfies
\begin{equation}\label{eps}
	\epsilon_{t+1}=(A-LC)\epsilon_t.
\end{equation}
The observer error $\epsilon_t=(A-LC)^t\epsilon_0$ converges to zero in finite time for every initial nonzero error if and only if $A-LC$ is nilpotent. If $A-LC$ is Schur but not nilpotent, the error converges asymptotically, although finite-time convergence may still occur for particular initial errors. To obtain an equivalent substitute for the state $x_t=x_t^c+\epsilon_t$, assertion (a) of Theorem \ref{thm2} provides a model-free parameterization of $x^c_t$ and $\epsilon_t$. Under the controllability of system (\ref{system}), assertion (b) of Theorem \ref{thm2} further shows that the corresponding state parameterization matrix has full row rank.

To facilitate the subsequent derivation, we first introduce the following notations.
Let the characteristic polynomial of $A-LC$ be denoted as $\Lambda(z):=\mathrm{det}(zI_n-(A-LC)):=z^n+a_{n-1}z^{n-1}+\cdots+a_1z+a_0$, and the adjoint matrix of $zI_n-(A-LC)$  be denoted as $\mathrm{adj}(zI_n-(A-LC)):=D_{n-1}z^{n-1}+\cdots+D_1z+D_0$ with $D_i=\sum_{j=i+1}^{n}a_{j}(A-LC)^{j-i-1}\in\mathbb{R}^{n\times n}, i=0,\cdots,n-1$. Set $a_n=1$ and define $\mathcal{D}:=[D_0,\cdots,D_{n-1}]$. 
Define $A_\epsilon$ as a square matrix that shares the same eigenvalues as $A-LC$, and replace $A-LC$ in the definition of $\mathcal{D}$ with $A_\epsilon$ to obtain $\mathcal{D}^{\epsilon}$ analogously. Let
	$$A_s:=\left[\begin{array}{c:c}
		\mathbf{0}&I_{n-1}\\\hdashline
		-a_0&\begin{matrix}
			-a_1&\cdots&-a_{n-1}
		\end{matrix}
	\end{array}\right], b_s:=\left[\begin{array}{c}
		\mathbf{0}_{(n-1)\times1}\\\hdashline1
	\end{array}\right], $$
and define $\mathcal{A}_s:=\mathrm{diag}(I_{m+p}\otimes A_s,A_\epsilon)$, $\mathcal{B}_s:=\mathrm{diag}(I_{m+p}\otimes b_s,\mathbf{0}_{n\times1})$.  
For $i=1,\cdots,m$ and $j=1,\cdots,p$, let $u_t^i$ and $y_t^j$ denote the $i$-th component of $u_t$ and the $j$-th component of $y_t$, respectively, with their z-transforms denoted as $U^i(z)$ and $Y^j(z)$. Let $B_i$ and $L_j$ denote the $i$-th column of $B$ and the $j$-th column of $L$, respectively. Define $S_u^i:=\mathcal{D}(I_n\otimes B_i)$, $S_y^j:=\mathcal{D}(I_n\otimes L_j)$, and
\begin{equation}\label{eu}    \eta_t^{u^i}:=\mathcal{Z}^{-1}\left(\left[
	     \frac{U^i(z)}{\Lambda(z)},\frac{zU^i(z)}{\Lambda(z)},\cdots,\frac{z^{n-1}U^i(z)}{\Lambda(z)}
	     \right]^\top\right),
\end{equation}
   \begin{equation}\label{ey}      \eta_t^{y^j}:=\mathcal{Z}^{-1}\left(\left[
	     \frac{Y^j(z)}{\Lambda(z)},\frac{zY^j(z)}{\Lambda(z)},\cdots,\frac{z^{n-1}Y^j(z)}{\Lambda(z)}
	     \right]^\top\right),
   \end{equation}
         where $\eta_t^{u^i}$ and $\eta_t^{y^j}$ satisfy systems $\eta_{t+1}^{u^i}=A_s\eta_{t}^{u^i}+b_su^i_t$, $\eta_{0}^{u^i}=\mathbf{0}$ and $\eta_{t+1}^{y^j}=A_s\eta_{t}^{y^j}+b_sy^j_t$, $\eta_{0}^{y^j}=\mathbf{0}$, respectively. Denote $\eta_t^u:=[(\eta_t^{u^1})^\top,\cdots,(\eta_t^{u^m})^\top]^\top$, $\eta_t^y:=[(\eta_t^{y^1})^\top,\cdots,(\eta_t^{y^p})^\top]^\top$, $S_u:=[S_u^1,\cdots,S_u^m]$,  $S_y:=[S_y^1,\cdots,S_y^p]$.

\begin{theorem}
\label{thm2}
Consider system (\ref{system}). If Assumption \ref{as1} holds, then the following assertions hold.
	\begin{itemize}
		\item[\textup{(a)}] 
		The state $x_t$ of system (\ref{system}) can be expressed in terms of the input-output filter vector $\eta_t\in\mathbb{R}^{(m+p+1)n}$ as 
		\begin{equation}\label{x_t4}
			x_t=[S_u,S_y,S_{\epsilon}][(\eta_t^u)^\top,(\eta_t^y)^\top,(\eta_t^{\epsilon})^\top]^\top:=S\eta_t,
		\end{equation}
		where  $\eta_t^{\epsilon}$ satisfies the system
	     \begin{equation}\label{A_e}
	 	\eta_{t+1}^{\epsilon}=A_{\epsilon}\eta_t^{\epsilon}, \quad \eta_0^{\epsilon}\neq\mathbf{0},
	     \end{equation}
         and $S_{\epsilon_x}:=\mathcal{D}(I_n\otimes\epsilon_0)$, $S_{\epsilon_{\eta}}:=\mathcal{D}^{\epsilon}(I_n\otimes\eta_0^\epsilon)$, $S_{\epsilon}:=S_{\epsilon_x}S_{\epsilon_{\eta}}^{-1}$.
The substitute state $\eta_t$ can be generated in a model-free manner via the following system:
		\begin{equation}\label{S0}
    \begin{split}
    \eta_{t+1}&=\mathcal{A}_s\eta_t+\mathcal{B}_s\left[
				u_t^\top, y_t^\top
				,\mathbf{0}_{1\times n}
			\right]^\top,\quad
            \eta_{0}=[\mathbf{0}_{1\times (m+p)n}, (\eta_{0}^{\epsilon})^\top]^\top.
    \end{split}
		\end{equation}
		\item[\textup{(b)}] The state parameterization matrix $S$ has full row rank.
\end{itemize}
\end{theorem}
\begin{proof}
	The proof of assertion (a) can be found in \cite{Deng-DT-output}. We now present the proof of assertion (b).
    Note that in equation (\ref{S0}), $\eta_t^u$ and $\eta_t^y$ are obtained by operating on each component of $u_t$ and $y_t$, respectively. Here, to facilitate theoretical analysis, the operations are performed on the vectors as a whole. Specifically, we consider
   \begin{align}\label{sig}
\sigma_t:=P^{\eta}\eta_{t}=\left[\mathcal{Z}^{-1}\left(\left[\frac{U(z)}{\Lambda(z)}^\top,\cdots,\frac{z^{n-1}U(z)}{\Lambda(z)}^\top, \frac{Y(z)}{\Lambda(z)}^\top,	
       \cdots,\frac{z^{n-1}Y(z)}{\Lambda(z)}^\top\right]^\top\right),(\eta_{t}^{\epsilon})^\top\right]^\top,
    \end{align}
    where $P^{\eta}$ is a row permutation matrix. Then, we obtain
   $x_t=S(P^{\eta})^\top\sigma_t=\mathcal{D}[I_n\otimes B, I_n\otimes L,(I_n\otimes\epsilon_0)(\mathcal{D}^\epsilon(I_n\otimes\eta_{0}^\epsilon))^{-1}]\sigma_t$ and we only need to consider the row rank of $S(P^{\eta})^\top$.
     Since $\mathcal{D}^{\epsilon}(I_n\otimes\eta_0^{\epsilon})$ is a user-defined invertible matrix and each block of $\mathcal{D}$ is a linear combination of $(A-LC)^{i}$ for $i=0,\cdots, n-1$, with invertible coefficient matrix, we only need to consider the row rank of $S^{\prime}:=[(A-LC)^{n-1},\cdots,I_n](I_n\otimes[B,L,\epsilon_0])$. 
	By contradiction, suppose $S^{\prime}$ does not have full row rank. According to the Hautus criterion, there exists a nonzero vector $q \in\mathbb{C}^n$ such that $q^\top[(A-LC)-\lambda I_n,B,L,\epsilon_0]=\mathbf{0}$, where $\lambda$ is an eigenvalue of $A-LC$. This implies $q^\top(A-LC)=\lambda q^\top$, $q^\top B=\mathbf{0}$ and $q^\top L=\mathbf{0}$. It follows that $q^\top A=\lambda q^\top$ and hence $q^\top[A-\lambda I_n,B]=\mathbf{0}$, which contradicts the controllability of system (\ref{system}). Therefore, $S^{\prime}$ has full row rank, and so does $S$.
\end{proof}

\begin{remark}\label{re3}	
    The controllability of system (\ref{system}) is a sufficient condition for $S$ to have full row rank.
    Suppose system (\ref{system}) is stabilizable but uncontrollable.
    To ensure that $S$ still has full row rank, one can set the eigenvalues of $A-LC$ to be different from those of $A$, ensuring $\mathcal{D}(I_n\otimes L)$ has full row rank \cite{State-Para}. 
\end{remark}

\begin{remark}\label{re4}
When $N=n$, the state parameterization (\ref{x_t}) is a special case of (\ref{x_t4}). Specifically, set all the eigenvalues of $A-LC$ to zero in Theorem \ref{thm2}. For $t\geq n$, remove the autonomous subsystem (\ref{A_e}) in Theorem \ref{thm2}. Then $\xi_t=P^{\prime}[(\eta_t^u)^\top,(\eta_t^y)^\top]^\top$ and $x_t=M\xi_t=[S_u,S_y][(\eta_t^u)^\top,(\eta_t^y)^\top]^\top$, where $P^\prime$ is a row permutation matrix.
Moreover, the state parameterization (\ref{x_t4}) is introduced for two reasons. First, it is a more general approach than (\ref{x_t}), and can be extended to continuous-time settings as described in Subsection \ref{section5_2}, whereas (\ref{x_t}) is a dead-beat method. Second, generating $\eta_{t+1}$ requires only the previous time-step input and output, rather than a long window of past trajectories.
\end{remark}

\begin{remark}
   Constructing $\eta_t$ requires no knowledge of $A$, $B$, $C$, or $L$. It suffices to choose the eigenvalues of $A-LC$ inside the unit circle. 
   Regarding the selection of $A_{\epsilon}$, it suffices to choose $A_{\epsilon}$ together with $\eta_0^{\epsilon}$ to ensure the invertibility of $S_{\epsilon_{\eta}}=\mathcal{D}^{\epsilon}(I_n\otimes \eta_0^{\epsilon})$. 
   For example, set $A_{\epsilon}=A_s$ and $\eta_0^\epsilon=b_s$; when $A-LC$ has distinct real poles, a diagonal $A_\epsilon$ and an $\eta_0^{\epsilon}$ with all entries nonzero constitute an alternative choice.
\end{remark}

The following Lemma \ref{le2} shows that the substitute state $\eta_t$ possesses properties analogous to those in Lemma \ref{le1}.
Different from Lemma \ref{le1}, we do not focus on the controllability of the consistent system (\ref{S1}), and directly explore the essential cause of rank deficiency for $H_0:=[\eta_0,\eta_1,\cdots,\eta_{T-1}]$. Specifically, we analyze the dimension of the minimal linear subspace $\mathcal{V}_{\eta}$ containing all possible $\eta_t$, where
$\mathcal{V}_{\eta}:=\{\eta_t=[(\eta_t^u)^\top,(\eta_t^y)^\top,(\eta_t^{\epsilon})^\top]^\top\vert \eta_t^u, \eta_t^y$ are defined by (\ref{eu}) and (\ref{ey}) for any input-output $u_t$, $y_t$ of system (\ref{system}), and
 $\eta_t^{\epsilon}$ is generated by system (\ref{A_e})$\}$.

\begin{lemma}\label{le2}
	If Assumption \ref{as1} holds and the substitute state $\eta_{t}$ is generated by equation (\ref{S0}), the following assertions hold.
	\begin{itemize}
		\item[\textup{(a)}] The relationship between the state $x_t$ of system (\ref{system}) and the input-output filter vector $\eta_{t}$ can be expressed by 
		$x_t=(S+S_0)\eta_t$,
		where $S_0$ is defined in equation (\ref{def-s}).
		\item[\textup{(b)}] System 
		\begin{equation}\label{S1}
			\begin{split}
			    \eta_{t+1}&=(\mathcal{A}_s+\mathcal{B}_s^yC(S+S_0))\eta_t+\mathcal{B}_s^uu_t, \\
                y_t&=C(S+S_0)\eta_t
			\end{split}
		\end{equation}
	    is consistent with system (\ref{system}), where $\mathcal{B}_s^u$ consists of the first $m$ columns of $\mathcal{B}_s$, and $\mathcal{B}_s^y$ consists of columns $m+1$ to $m+p$ of $\mathcal{B}_s$.
		\item[\textup{(c)}] Suppose not all eigenvalues of $A-LC$ are assigned to zero and $\eta_t$ does not reduce to $\xi_t$. Assume the inputs of system (\ref{system}) satisfy the $2n$-th order PE condition. Then the data matrix $H_0:=[\eta_{0},\dots,\eta_{T-1}]$ has rank $mn+2n$, and $H_0$ has full row rank if and only if $p=1$.
	\end{itemize}
\end{lemma}

\begin{proof}
    Let $d$ denote the dimension of $\mathcal{V}_{\eta}$, $f:=(m+p+1)n$. Choose a matrix $\mathcal{P}\in\mathbb{R}^{f\times d}$ whose columns form an orthonormal basis for $\mathcal{V}_{\eta}$, and select $\mathcal{P}_{\perp}\in\mathbb{R}^{f\times(f-d)}$ whose columns form an orthonormal basis for the orthogonal complement of $\mathcal{V}_{\eta}$. %so that $P_{\perp}^\top P=\mathbf{0}$
    For any $\mathcal{N}\in\mathbb{R}^{n\times (f-d)}$, define
    \begin{equation}\label{def-s}
        S_0=\mathcal{N}\mathcal{P}_{\perp}^\top.
    \end{equation}
    For any $\eta\in\mathcal{V}_{\eta}$, we have $S_0\eta=\mathbf{0}$. This establishes the existence of $S_0$. Combined with Theorem \ref{thm2}, assertions (a) and (b) follow readily.
    
    For assertion (c), we now prove that the dimension of $\mathcal{V}_{\eta}$ is $mn+2n$ under the assumption of assertion (c). Since $S_0$ annihilates $\mathcal{V}_{\eta}$, we choose $S_0=\mathbf{0}$. Denote $\tilde{\mathcal A}:=\mathcal A_s+\mathcal B_s^yCS$.
    We first prove that from the the special structures of $\mathcal{A}_s$, $\mathcal{B}_s$ and $S$, we can obtain
	$$S\mathcal{A}_s=(A-LC)S,\qquad
	SB_s^u=B,\qquad
	S\mathcal{B}_s^y=L, S\tilde{\mathcal{A}}=AS.$$
	By the definition of the adjugate matrix, $(zI_n-(A-LC))\mathrm{adj}(zI_n-(A-LC))=\Lambda(z)I_n$. Substituting the expansions of the characteristic polynomial and the adjoint matrix yields
	$$\sum_{i=0}^{n-1}D_iz^{i+1}-\sum_{i=0}^{n-1}(A-LC)D_iz^i=I_n z^n+\sum_{i=0}^{n-1}a_iI_nz^i.$$
	Comparing the coefficients, we obtain $D_{n-1}=I_n$, $(A-LC)D_0=-a_0I_n$, $(A-LC)D_i=D_{i-1}-a_iI_n$, $i=1,\cdots,n-1$.
	Furthermore, from the structure o $A_s$, for any $b\in\mathbb{R}^n$,
	$$[D_0b,D_1b,\cdots,D_{n-1}b]A_s=[-a_0b,D_0b-a_1b,\cdots,D_{n-2}b-a_{n-1}b],$$
	where we have used $D_{n-1}b=b$.
	On the other hand,
	$$[(A-LC)D_0b,(A-LC)D_1b,\cdots,(A-LC)D_{n-1}b]=[-a_0b,D_0b-a_1b,\cdots,D_{n-2}b-a_{n-1}b].$$
	Hence
	$$[D_0b,D_1b,\cdots,D_{n-1}b]A_s=[(A-LC)D_0b,(A-LC)D_1b,\cdots,(A-LC)D_{n-1}b].$$
	Since $b_s$ is the vector taken as the last column,
	$$[D_0b,D_1b,\cdots,D_{n-1}b]b_s=D_{n-1}b=b.$$
	Setting $b=B_i$ and $b=L_j$ separately and concatenating the resulting blocks, we have
	\begin{equation}\label{*}
		\begin{split}
		&S_u(I_m\otimes A_s)=(A-LC)S_u,\quad S_u(I_m\otimes b_s)=B,\\
		&S_y(I_p\otimes A_s)=(A-LC)S_y,\quad S_y(I_p\otimes b_s)=L.
		\end{split}\tag{*}
	\end{equation}
	Define $K_{o}=[\epsilon_{0},(A-LC)\epsilon,\cdots,(A-LC)^{n-1}\epsilon_{0}]$ and $K_{\epsilon}=[\eta_0^{\epsilon},A_{\epsilon}\eta_0^{\epsilon},\cdots,A_{\epsilon}^{n-1}\eta_0^{\epsilon}]$. t is easy to see that there exists an invertible matrix $T_{\Lambda}$, depending only on the coefficients of $\Lambda(z)$, such that $S_{\epsilon_x}=K_{o}T_{\Lambda}$, $S_{\epsilon_{\eta}}=K_{\epsilon}T_{\Lambda}$, and $S_{\epsilon}=K_{o}K_{\epsilon}^{-1}$.
	Define
	$$J_{\Lambda}=\begin{bmatrix}
		0&0&\cdots&-a_0\\1&0&\cdots&-a_1\\0&1&\cdots&-a_2\\\vdots&\vdots&\ddots&\vdots\\
		0&0&\cdots&-a_{n-1}
	\end{bmatrix}.$$
	By the Cayley-Hamilton theorem, $(A-LC)^n=-\sum_{i=0}^{n-1}a_i(A-LC)^i$, which implies
	$(A-LC)K_o=K_oJ_{\Lambda}$. Similarly, $K_{\epsilon}^{-1}A_{\epsilon}=J_{\Lambda}K_{\epsilon}^{-1}$. Therefore,	$$S_{\epsilon}A_{\epsilon}=K_oK_{\epsilon}^{-1}A_{\epsilon}=K_oJ_{\Lambda}K_{\epsilon}^{-1}=(A-LC)K_oK_{\epsilon}^{-1}=(A-LC)S_{\epsilon}.$$
	Combining this identity with equation (*), we obtain
	\begin{align*}
		S\mathcal{A}_{s}=[S_u(I_m\otimes \mathcal{A}_s), S_y(I_p\otimes A_s),S_{\epsilon}A_{\epsilon}]=[(A-LC)S_u,(A-LC)S_y,(A-LC)S_{\epsilon}]=(A-LC)S,
	\end{align*}
	$$S\mathcal{B}_s^u=S_u(I_m\otimes b_s)=B,\quad S\mathcal{B}_s^y=S_y(I_p\otimes b_s)=L.$$
    Then,
	\begin{align*}
		S\tilde{\mathcal{A}}=S\mathcal{A}_s+S\mathcal{B}_s^yCS=(A-LC)S+LCS=AS.
	\end{align*}

	We next prove that the dimension of $\mathcal{V}_{\eta}$ is $mn+2n$ under the assumption of assertion (c).
	Since $S_0$ annihilates $\mathcal{V}_{\eta}$, we choose $S_0=\mathbf{0}$. Denote $\tilde{\mathcal A}:=\mathcal A_s+\mathcal B_s^yCS$.
	From the definition of $\mathcal{V}_{\eta}$, Theorem \ref{thm2} and assertion (b), $\mathcal{V}_{\eta}$ equals the sum of the input-reachable subspace of system (\ref{S1}), $\mathcal{V}:=\mathrm{im}([\mathcal{B}_s^u,\tilde{\mathcal{A}}\mathcal{B}_s^u,\cdots,\tilde{\mathcal{A}}^{f-1}\mathcal{B}_s^u])$, and $\mathrm{im}[\eta_0,\tilde{\mathcal{A}}\eta_0,\cdots,\tilde{\mathcal{A}}^{f-1}\eta_0]$. We now compute the dimension of $\mathcal{V}$. 
	$\mathcal{V}$ characterizes the state difference directions induced by arbitrary inputs and zero input under identical initial conditions, whose autonomous coordinates are identically zero. That is, the difference $\Delta \eta_t$ satisfies
	$$\Delta \eta_{t+1}=\tilde{\mathcal{A}}\Delta\eta_t+\mathcal{B}_s^u\Delta  u_t,\quad \Delta \eta_0=\mathbf{0}.$$
	Denote the first $(m+p)n$ difference coordinates by $\delta_t=P_f\Delta \eta_t$, where $P_f=[I_{(m+p)n},\mathbf{0}]$.
	Define $v_n(z):=[1,z,\ldots,z^{n-1}]^\top$,	$g(z):=\frac{v_n(z)}{\Lambda(z)}$. The zero-initial-condition transfer matrix from the input to $\delta_t$ is
	$$\mathcal{T}(z)=
	\begin{bmatrix}
		I_m\otimes g(z)\\
		(I_p\otimes g(z))C(zI_n-A)^{-1}B
	\end{bmatrix}.$$
	Take an arbitrary row vector  $q=[q_{u,1},\cdots,q_{u,m},q_{y,1},\cdots,q_{y,p}]$, where each block contains $n$ entries. 
	Then $q$ annihilates all difference responses if and only if $q\mathcal{T}(z)=\mathbf{0}$.
	Define $\alpha(z):=[q_{u,1}v_{n}(z),\cdots,q_{u,m}v_n(z)]$ and $\beta(z):=[q_{y,1}v_n(z),\cdots,q_{y,p}v_n(z)]$. Hence, $[q_{u,1},\cdots,q_{u,m}](I_m\otimes g(z))=\alpha(z)/\Lambda(z)$ and $[q_{y,1},\cdots,q_{y,p}](I_m\otimes g(z))=\beta(z)/\Lambda(z)$. Thus,
	$$q\mathcal{T}(z)=\frac{\alpha(z)+\beta(z)C(zI_n-A)^{-1}B}{\Lambda(z)}.$$
	Since $\Lambda(z)$ is a nonzero polynomial, $q\mathcal{T}(z)=\mathbf{0}$ if and only if
	$$\alpha(z)+\beta(z)C(zI_n-A)^{-1}B=\mathbf{0}.$$
	Write $\beta(z)=\sum_{k=0}^{n-1}\beta_kz^k$, $\beta_k\in\mathbb{R}^{1\times p}$.
	From $zI_n=(zI_n-A)+A$, we have $z(zI_n-A)^{-1}=I_n+A(zI_n-A)^{-1}$. Repeated application of this identity yields
	$$z^k(zI_n-A)^{-1}=\sum_{j=0}^{k-1}z^{k-1-j}A^j+A^k(zI_n-A)^{-1}.$$
	Substituting $\beta(z)$ and using above identity, we obtain
	\begin{align*}
		\beta(z)C(zI_n-A)^{-1}&=\sum_{k=0}^{n-1}\beta_kC[z^k(zI_n-A)^{-1}]\\
		&=\sum_{k=1}^{n-1}\sum_{j=0}^{k-1}\beta_kCA^jz^{k-1-j}+\left(\sum_{k=0}^{n-1}\beta_kCA^k\right)(zI_n-A)^{-1}.
	\end{align*}
	Define $Q_{\beta}(z):=\sum_{k=1}^{n-1}\sum_{j=0}^{k-1}\beta_kCA^jz^{k-1-j}$ and $r_{\beta}:=\sum_{k=0}^{n-1}\beta_kCA^k$. The identity becomes
	$$\alpha(z)+Q_{\beta}(z)B+r_{\beta}(zIn-A)^{-1}B=\mathbf{0},$$
	where the strictly proper rational part $r_{\beta}(zIn-A)^{-1}B$ tends to $\mathbf{0}$ as $z\rightarrow+\infty$. Accordingly, the polynomial row vector $\alpha(z)+Q_{\beta}(z)B$ must also vanish. A polynomial that tends to zero as $z\rightarrow+\infty$ is necessarily the zero polynomial. Applying this argument componentwise yields $\alpha(z)+Q_{\beta}(z)B\equiv \mathbf{0}$, which further implies
	$$r_{\beta}(zIn-A)^{-1}B=\left(\sum_{k=0}^{n-1}\beta_kCA^k\right)
	(zI_n-A)^{-1}B\equiv \mathbf{0}.$$	
	By controllability of $(A,B)$, this is equivalent to $[\beta_0,\cdots,\beta_{n-1}]\mathcal{O}(A,C)=\mathbf{0}$.
	By observability of $(A,C)$, the set of admissible coefficient tuples $\{\beta_0\cdots,\beta_{n-1}\}$ form a vector space of dimension $(p-1)n$. Note that each such tuple uniquely determines $\alpha(z)$. 
	Thus, the reachability matrix associated with the first $(m+p)n$ difference coordinates has a $(p-1)n$-dimensional left kernel space. Since the autonomous coordinates of the full difference state vanish, the dimension of $\mathcal{V}$ is $(m+p)n-(p-1)n=mn+n$.
	
	We now prove that $\mathrm{im}([\eta_0,\tilde{\mathcal{A}}\eta_0,\cdots,\tilde{\mathcal{A}}^{f-1}\eta_0])$ has dimension $n$ and equals the orthogonal complement of $\mathcal{V}$ within $\mathcal{V}_{\eta}$.
	Note that
	$S\tilde{\mathcal{A}}=AS$ and $S\mathcal{B}_s^u=B$, which implies $S\tilde{\mathcal{A}}\mathcal{B}_s^u=AS\mathcal{B}_s^u=AB$ and $S\tilde{\mathcal{A}}^2\mathcal{B}_s^u=A(S\tilde{\mathcal{A}})\mathcal{B}_s^u=A^2(S\mathcal{B}_s^u)=A^2B$. By induction, $S\tilde{\mathcal{A}}^j\mathcal{B}_s^u=A^jB$ holds for all $j\geq0$. Therefore,
	$$\mathrm{im}(S|_{\mathcal{V}})=\mathrm{im}(S[\mathcal{B}_s^u,\tilde{\mathcal{A}}\mathcal{B}_s^u,\cdots,\tilde{\mathcal{A}}^{f-1}\mathcal{B}_s^u])=\mathrm{im}([B,AB,\cdots,A^{f-1}B]).$$
	Furthermore, the controllability of $(A,B)$ yields $\mathrm{im}(S|_{\mathcal{V}})=\mathbb{R}^n$. 
	Take any $w\in\mathrm{ker}(S)$, i.e., $Sw=\mathbf{0}$. Using $S\tilde{\mathcal{A}}=AS$, we have $S\tilde{\mathcal{A}}w=ASw=A\mathbf{0}=\mathbf{0}$, thereby $\tilde{\mathcal{A}}w\in\mathrm{ker}(S)$. Moreover, $\tilde{\mathcal{A}}w=\mathcal{A}_sw+\mathcal{B}_s^yCSw=\mathcal{A}_sw$. Hence, the subspace $\mathrm{ker}(S)$ is $\tilde{\mathcal{A}}$-invariant; $\tilde{\mathcal{A}}$ and $\mathcal{A}_s$ act identically on vectors in $\mathrm{ker}(S)$.
	Based on these two facts, for any $w\in\mathrm{ker}(S)$, we have $\tilde{\mathcal{A}}^{j}w=\mathcal{A}_s^jw$, $j\geq 0$. 
	Combined with $\Lambda(\mathcal{A}_s)=0$, it follows that $\Lambda(\tilde{\mathcal{A}})w=\Lambda(\mathcal{A}_s)w=\mathbf{0}$. 
	Then take arbitrary $r\in\mathbb{R}^f$ and select $h\in\mathcal{V}$ such that $Sh=Sr$.
	Then $r-h\in\mathrm{ker}(S)$ and $\Lambda(\tilde{\mathcal{A}})(r-h)=\mathbf{0}$. 
	Note that for any $e\in\mathcal{V}$, we can write $e=\sum_{j=0}^{f-1}\tilde{\mathcal{A}}^j\mathcal{B}_s^u k_j$. Left-multiplying by $\tilde{\mathcal{A}}$ yields $\tilde{\mathcal{A}}e=\sum_{j=0}^{f-1}\tilde{\mathcal{A}}^{j+1}\mathcal{B}_s^u k_j$. By the Cayley-Hamilton theorem, $\tilde{\mathcal{A}}^{f}$ can be expressed as a linear combination of $I,\tilde{\mathcal{A}},\cdots,\tilde{\mathcal{A}}^{f-1}$. Hence $\tilde{\mathcal{A}}e\in\mathcal{V}$, which implies that $\mathcal{V}$ is $\tilde{\mathcal{A}}$-invariant. Then,
	$\Lambda(\tilde{\mathcal{A}})r=\Lambda(\tilde{\mathcal{A}})h\in\mathcal{V}$.
	Since $f\geq n$, it is easy to see that $\mathcal{V}+\mathrm{im}([\eta_0,\tilde{\mathcal{A}}\eta_0,\cdots,\tilde{\mathcal{A}}^{n-1}\eta_0])\subseteq\mathcal{V}_{\eta}$. From the above results, setting $r=\eta_0$, we have $\tilde{\mathcal{A}}^n\eta_0+\sum_{i=0}^{n-1}a_i\tilde{\mathcal{A}}^i\eta_0\in\mathcal{V}$. Thus, there exists a vector $\phi\in\mathcal{V}$ such that 
	$\tilde{\mathcal{A}}^n\eta_0=\phi-\sum_{i=0}^{n-1}a_i\tilde{\mathcal{A}}^i\eta_0\in\mathcal{V}+\mathrm{im}([\eta_0,\tilde{\mathcal{A}}\eta_0,\cdots,\tilde{\mathcal{A}}^{n-1}\eta_0])$. This establishes the $\tilde{\mathcal{A}}$-invariance of $\mathcal{V}+\mathrm{im}([\eta_0,\tilde{\mathcal{A}}\eta_0,\cdots,\tilde{\mathcal{A}}^{n-1}\eta_0])$. Consequently, $\tilde{\mathcal{A}}^k\eta_0\in\mathcal{V}+\mathrm{im}([\eta_0,\tilde{\mathcal{A}}\eta_0,\cdots,\tilde{\mathcal{A}}^{n-1}\eta_0])$ for all $k\geq 0$. For any vector in $\mathcal{V}_{\eta}$,
	$\eta_{t}=\tilde{\mathcal{A}}^k\eta_{t-k}+\sum_{j=0}^{k-1}\tilde{\mathcal{A}}^{k-1-j}\mathcal{B}_s^y u_j$. From the preceding results, the first term on the right-hand side belongs to $\mathcal{V}+\mathrm{im}([\eta_0,\tilde{\mathcal{A}}\eta_0,\cdots,\tilde{\mathcal{A}}^{n-1}\eta_0])$, and the second term lies in the input reachable subspace $\mathcal{V}$. Hence, $\mathcal{V}_{\eta}\subseteq\mathcal{V}+\mathrm{im}([\eta_0,\tilde{\mathcal{A}}\eta_0,\cdots,\tilde{\mathcal{A}}^{n-1}\eta_0])$. It follows that $\mathcal{V}_{\eta}=\mathcal{V}+\mathrm{im}([\eta_0,\tilde{\mathcal{A}}\eta_0,\cdots,\tilde{\mathcal{A}}^{n-1}\eta_0])$.
	Let $P_\epsilon$ select the autonomous coordinates.
	Then $\mathrm{im}(P_\epsilon|_{\mathcal{V}})=\{\mathbf{0}\}$ and $P_\epsilon [\eta_0,\tilde{\mathcal{A}}\eta_0,\cdots,\tilde{\mathcal{A}}^{n-1}\eta_0]=[\eta_0^\epsilon,A_\epsilon\eta_0^\epsilon,\cdots,A_\epsilon^{n-1}\eta_0^\epsilon]$.
	The construction of $S_{\eta_\epsilon}$ gives $S_{\eta_\epsilon}=[\eta_0^\epsilon,A_\epsilon\eta_0^\epsilon,\cdots,A_\epsilon^{n-1}\eta_0^\epsilon]T_\Lambda$
	with $T_\Lambda$ nonsingular. Thus, $[\eta_0^\epsilon,A_\epsilon\eta_0^\epsilon,\cdots,A_\epsilon^{n-1}\eta_0^\epsilon]$ is nonsingular.
	Take any $w\in\mathcal{V}\cap\mathrm{im}([\eta_0,\tilde{\mathcal{A}}\eta_0,\cdots,\tilde{\mathcal{A}}^{n-1}\eta_0])$. Then, $w=[\eta_0,\tilde{\mathcal{A}}\eta_0,\cdots,\tilde{\mathcal{A}}^{n-1}\eta_0]c$ for some $c\in\mathbb{R}^n$.
	Since the autonomous coordinates of every vector
	in $\mathcal{V}$ vanish, $\mathbf{0}=P_\epsilon w=P_\epsilon [\eta_0,\tilde{\mathcal{A}}\eta_0,\cdots,\tilde{\mathcal{A}}^{n-1}\eta_0]c=[\eta_0^\epsilon,A_\epsilon\eta_0^\epsilon,\cdots,A_\epsilon^{n-1}\eta_0^\epsilon] c$. The nonsingularity of $[\eta_0^\epsilon,A_\epsilon\eta_0^\epsilon,\cdots,A_\epsilon^{n-1}\eta_0^\epsilon]$ implies $c=\mathbf{0}$,
	and hence $w=\mathbf{0}$. Therefore, $\mathcal{V}\cap\mathrm{im}([\eta_0,\tilde{\mathcal{A}}\eta_0,\cdots,\tilde{\mathcal{A}}^{n-1}\eta_0])=\{\mathbf{0}\}$. Moreover, $P_\epsilon [\eta_0,\tilde{\mathcal{A}}\eta_0,\cdots,\tilde{\mathcal{A}}^{n-1}\eta_0]=[\eta_0^\epsilon,A_\epsilon\eta_0^\epsilon,\cdots,A_\epsilon^{n-1}\eta_0^\epsilon]$ implies $\mathrm{rank}([\eta_0,\tilde{\mathcal{A}}\eta_0,\cdots,\tilde{\mathcal{A}}^{n-1}\eta_0])=n$. Thus, $\mathcal{V}_{\eta}$ has dimension $mn+2n$.
	
	We finally prove the rank of $H_0$ is $mn+2n$ under $2n$-th order PE inputs.
	It remains to show that the columns of $H_0$ spans $\mathcal{V}_{\eta}$.
	Define $\kappa(z):=\mathrm{det}(zI_n-A)$ and $\pi(z):=\Lambda(z)\kappa(z)=\sum_{i=0}^{2n}c_i z^i$ with $c_{2n}=1$.
	Since
	\begin{align*}
		S\kappa(\tilde{\mathcal{A}})=S\tilde{\mathcal{A}}^n+\sum_{j=0}^{n-1}c_jS\tilde{\mathcal{A}}^j=A^nS+\sum_{j=0}^{n-1}c_jA^jS=\kappa(A)S=0,
	\end{align*}
	it holds that $\mathrm{im}(\kappa(\tilde{\mathcal{A}}))\subseteq\mathrm{ker}(S)$.
	And since $\Lambda(\tilde{\mathcal{A}})$ vanishes on $\mathrm{ker}(S)$, for any $c\in\mathbb{R}^f$, it holds that $\kappa(\tilde{\mathcal{A}})c\in\mathrm{ker}(S)$. Then $\pi(\tilde{\mathcal{A}})c=\Lambda(\tilde{\mathcal{A}})(\kappa(\tilde{\mathcal{A}})c)=\mathbf{0}$. Since this holds for all $c$, we conclude that $\pi(\tilde{\mathcal{A}})=\mathbf{0}$.	
	Take $w$ such that $w^\top H_0=0$, which means $w^\top\eta_t=0$, $t=0,\cdots,T-1$.
	Expanding the state recursion gives
	$$\eta_{t+i}=\tilde{\mathcal{A}}^i\eta_t+\sum_{k=0}^{i-1}\tilde{\mathcal{A}}^{i-1-k}\mathcal{B}_s^uu_{t+i}.$$
	Multiply by $c_i$ and sum over $i=0,\cdots,2n$, which gives
	$$\sum_{i=0}^{2n}c_i\eta_{t+i}=\left(\sum_{i=0}^{2n}c_i\tilde{\mathcal{A}}^i\right)\eta_t+\sum_{i=1}^{2n}\sum_{k=0}^{i-1}c_i\tilde{\mathcal{A}}^{i-1-k}\mathcal{B}_s^uu_{t+i}.$$
	We have shown earlier that $\pi(\tilde{\mathcal{A}})=\sum_{i=0}^{2n}c_i\tilde{\mathcal{A}}^i=\mathbf{0}$, thus
	$$\sum_{i=0}^{2n}c_i\eta_i=\sum_{k=0}^{2n-1}\left(\sum_{i=k+1}^{2n}c_i\tilde{\mathcal{A}}^{i-1-k}\mathcal{B}_s^u\right)u_{t+k}.$$
	Left-multiplying by $w^\top$ and collecting the equations for
	$t=0,\ldots,T-2n-1$ gives $[\mathcal{\chi}_0,\cdots,\mathcal{\chi}_{2n-1}]
	\mathcal{H}_{2n}(u_{[0,T-2]})=\mathbf{0}$, where $\mathcal{\chi}_k:=
	w^\top\sum_{i=k+1}^{2n}
	c_i\tilde{\mathcal{A}}^{i-1-k}\mathcal{B}_s^u$.
	By the $2n$-th order PE condition, $\chi_k=\mathbf{0}$. 
	Backward substitution yields
	$w^\top\tilde{\mathcal{A}}^{j}\mathcal{B}_s^u=\mathbf{0}$, $j=0,\cdots,2n-1$.
	The identity $\pi(\mathcal A)=\mathbf{0}$ implies $\tilde{\mathcal{A}}^{2n}=-\sum_{i=0}^{2n-1}c_i\tilde{\mathcal{A}}^i$, $w^\top\tilde{\mathcal{A}}^{2n}\mathcal{B}_s^u=-\sum_{i=0}^{2n-1}c_iw^\top\tilde{\mathcal{A}}^i\mathcal{B}_s^u=\mathbf{0}$.
	More generally, multiplying the polynomial identity by $\tilde{\mathcal{A}}^k$ yields $\tilde{\mathcal{A}}^{k+2n}=-\sum_{i=0}^{2n-1}c_i\tilde{\mathcal{A}}^{k+i}$, and hence $w^\top\tilde{\mathcal{A}}^{k+2n}\mathcal{B}_s^u=-\sum_{i=0}^{2n-1}c_iw^\top\tilde{\mathcal{A}}^{k+i}\mathcal{B}_s^u$. By induction, $w^\top \tilde{\mathcal{A}}^{j}\mathcal{B}_s^u=\mathbf{0}$ for every $j\geq0$. Therefore, $w^\top\mathcal{V}=\{\mathbf{0}\}$.
	Furthermore, from $w^\top \tilde{\mathcal{A}}^{j}\mathcal{B}_s^u=\mathbf{0}$, the state recursion gives $w^\top\eta_{j}=w^\top\tilde{\mathcal{A}}^j\eta_0=0$, for $j=0,\cdots,T-1$.
	Using again
	$\pi(\tilde{\mathcal{A}})=\mathbf{0}$, the equality $w^\top\tilde{\mathcal{A}}^j\eta_0=0$ can be extended to all $j\geq0$.
	Thus, $w$ also annihilates $\mathrm{im}([\eta_0,\tilde{\mathcal{A}}\eta_0,\cdots,\tilde{\mathcal{A}}^{f-1}\eta_0])$.
	Since every column of $H_0$ belongs to
	$\mathcal{V}_{\eta}$, equality of the left annihilators
	implies $\mathrm{im}(H_0)=\mathcal{V}_{\eta}$.
	Therefore $\mathrm{rank}(H_0)=mn+2n$.
\end{proof}

\subsection{Projection-Based State Parameterization}\label{section3_3}
From Remark \ref{re4}, Subsection \ref{section3_1} can be viewed as a special case of Subsection \ref{section3_2}. We thus adopt a unified framework for both subsections: set $N=n$ for Subsection \ref{section3_1}; the two state parameterizations (\ref{x_t}) and (\ref{x_t4}) are uniformly written as $x_t=E \zeta_t$, where $n_{\zeta}$ is the dimension of $\zeta_t$; unify the two substitute state data matrices $\Xi_0$ and $H_0$ as $Z_0:=[\zeta_0,\cdots,\zeta_{T-1}]$, with row-rank $n_{v}$ under PE condition; the two spaces $\mathcal{V}_{\xi}$ and $\mathcal{V}_{\eta}$ are unified as $\mathcal{V}_{\zeta}$;
the two consistent systems (\ref{M_0}) and (\ref{S1}) are represented as $(\mathcal{A},\mathcal{B},\mathcal{C})$.

This subsection constructs an improved substitute state $v_t=G\zeta_t$ such that the corresponding data matrix $V_0:=[v_0,\cdots,v_{T-1}]$ has full row rank under MO settings, thereby achieving the objective stated at the beginning of Section \ref{section3}.

\begin{lemma}\label{le4}
Given $Z_0:=[\zeta_0,\cdots,\zeta_{T-1}]$ obtained under the PE condition, suppose that the projection matrix $G\in\mathbb{R}^{n_{v} \times n_{\zeta}}$ satisfies
	\begin{equation}\label{imP}
		\mathrm{im}(G^\top)=\mathcal{V}_{\zeta}=\mathrm{im}(Z_0).
	\end{equation}
	 Define $v_t=G\zeta_t\in\mathbb{R}^{n_v}$. Then the following assertions hold.
    \begin{itemize}
        \item[(a)] $v_t$ can serve as a substitute state for the unmeasurable $x_t$ of system (\ref{system}), i.e., $x_t=EG^\top(GG^\top)^{-1}v_t:=Fv_t$, where the state parameterization matrix $F$ has full row rank.
        \item[(b)] The data matrix $V_0:=[v_0,\cdots,v_{T-1}]$ is guaranteed to have full row rank.
    \end{itemize}
\end{lemma}

\begin{proof}
By (\ref{imP}), $G^{\dagger}=G^\top(GG^\top)^{-1}$, and for any $\zeta\in\mathcal{V}_{\zeta}$, there exists $a\in\mathbb{R}^{n_{v}}$ such that $\zeta=G^\top a$. Thus, $G^{\dagger}G\zeta=\zeta$. Substituting $v_t=G\zeta_t$ gives $\zeta_t=G^\dagger v_t$ and $x_t=EG^\dagger v_t=Fv_t$. Furthermore, $\mathrm{im}(G^\dagger)=\mathcal{V}_\zeta$. Since the proofs of Lemmas \ref{le1} and \ref{le2} establish $E\mathcal{V}_\zeta=\mathbb R^n$, we have $\mathrm{im}(F)=E\mathrm{im}(G^\dagger)=E\mathcal{V}_\zeta=\mathbb R^n$.
Thus $\mathrm{rank}(F)=n$, proving assertion (a). 
Since every column of $Z_0$ belongs to $\mathcal{V}_\zeta$, $Z_0=G^\dagger GZ_0=G^\dagger V_0$. Consequently, $n_v=\mathrm{rank}(Z_0)
\leq\mathrm{rank}(V_0)\leq n_v$. 
Therefore $\mathrm{rank}(V_0)=n_v$, proving
assertion (b).
\end{proof}

Matrix $G$ maps vectors in $\mathcal{V}_{\zeta}$ to an $n_v$-dimensional coordinate representation and can be implemented in a model-free manner via the following methods. (i) Row selection method: Let $\tilde{R}$ be the row-selection matrix extracting $n_v$ linearly independent rows of $Z_0$. Set $G=\tilde{R}Z_0Z_0^\dagger$. Then, $V_0=GZ_0=\tilde{R}(Z_0Z_0^\dagger Z_0)=\tilde{R}Z_0$. (ii) QR decomposition method: 
Compute the rank-$n_v$ column-pivoted QR factorization $Z_0\Pi_c=\tilde{Q}\tilde{R}$, where
$\Pi_c\in\mathbb R^{T\times T}$ is a column permutation
matrix, $\tilde{Q}^\top \tilde{Q}=I_{n_v}$, and $\tilde{R}\in\mathbb R^{n_v\times T}$ has full row rank. Choose $G=\tilde{Q}^\top$, yielding $V_0=\tilde{R}\Pi_c^\top$.
(iii) Singular value decomposition (SVD) method: Compute the compact SVD $Z_0=U_v\Sigma_vW_v^\top$, where $\Sigma_v\in\mathbb{R}^{n_v\times n_v}$.
Choose $G=U_v^\top$, yielding
$V_0=\Sigma_vW_v^\top$.
If $\zeta_t=\xi_t$, $G$ provides reduced coordinates on the reachable subspace of the consistent system (\ref{M_0}). 
A particular row selection retains all input-history components and selects output-history components corresponding to linearly independent rows of the $N$-step observability matrix. This follows the same selection principle as in \cite{Notes-MO} for systems without direct feedthrough.

According to Lemma \ref{le4}, $v_t$ further reduces redundant information as well as the dimension. Meanwhile, the following Theorem \ref{thm3} provides two conclusions that facilitate algorithm design. First, the data matrix has full row rank, which ensures the feasibility of the algorithms in Section \ref{section4}. Second, the feedback controller $u_t=Kv_t$ can be directly parameterized by the data matrix. Therefore, we refer to the improved substitute state $v_t$ as data-efficient.

\begin{theorem}\label{thm3} 
If the input sequence $\{u_t\}$ of system (\ref{system}) satisfies PE of order $2n+1$, the following assertions hold.
\begin{itemize}
    \item[(a)] For both SO and MO problems, the matrix $[V_0^\top,U_0^\top]^\top$ generated under Lemma \ref{le4} has full row rank.
    \item[(b)] For the feedback controller $u_t=Kv_t$, it holds that
\begin{equation}\label{PK}
    \mathrm{im}\left([(G^\dagger)^\top,K^\top]^\top\right)\subseteq \mathrm{im}\left([Z_0^\top,U_0^\top]^\top\right),
\end{equation}
which implies that there exists a matrix 
 $\Phi\in\mathbb{R}^{T\times n_v}$ such that $G^\dagger=Z_0\Phi$ (in particular, $I_{n_v}=V_0\Phi$), and the feedback gain $K$ can be parameterized by the data matrix $U_0$ as $K=U_0\Phi$.
\end{itemize}
\end{theorem}

In the algorithm design of Section \ref{section4}, we primarily employ assertion (a). Equation (\ref{PK}) serves as the counterpart of (\ref{IK}) when the state is unmeasurable. Without projection, directly using $u_t=K_{\zeta}\zeta_t$ as the feedback controller would not yield this result. Moreover, it should be noted that the PE orders used in this paper are always sufficient conditions.

\begin{proof}
   Note that $v_t$ satisfies $v_{t+1}=G\mathcal{A}G^\dagger v_t+G\mathcal{B}u_t$, $y_t=CEG^\dagger v_t$. From Remark \ref{re4} and the proof of Lemma \ref{le2} (c), for any $\zeta\in\mathcal{V}_{\zeta}$, we have $E\mathcal{A}\zeta=AE\zeta$, and $\mathcal{V}_{\zeta}$ is $\mathcal{A}$-invariant. For any $a\in\mathbb{R}^{n_v}$, $G^\dagger a\in\mathcal{V}_{\zeta}$ and $\mathcal{A}G^\dagger a\in\mathcal{V}_{\zeta}$. Since $G^\dagger G$ is the orthogonal projection onto $\mathcal{V}_{\zeta}$, it follows that $G^\dagger G\mathcal{A}G^\dagger a=\mathcal{A}G^\dagger a$. Hence,
	$$FG\mathcal{A}G^\dagger a=EG^\dagger G\mathcal{A}G^\dagger a=E\mathcal{A}G^\dagger a=AE G^\dagger a=AFa.$$
	As $a$ is arbitrary,
	$F\mathcal G\mathcal{A}G^\dagger=AF$.
    
	We next prove that $\Lambda(G\mathcal{A}G^\dagger)w=\mathbf{0}$ for all $w\in\mathrm{ker}(F)$. 
	For the historical trajectory-based state parameterization with $N=n$, we have $x=Fw=\mathbf{0}$. Starting from this zero state and applying zero input, the states and outputs of both the original system (\ref{system}) and the consistent system (\ref{S1}) with state $v_t$ remain identically zero.
	Hence $(G\mathcal{A}G^\dagger)^n w=\mathbf{0}$, i.e., $\Lambda(G\mathcal{A}G^\dagger)w=\mathbf{0}$.
	For the observer-based state parameterization, recall from earlier that $\Lambda(\mathcal{A})h=\mathbf{0}$ for any $h\in\mathrm{ker}(E)$. For each $w\in\mathrm{ker}(F)$, there exists $h\in\mathrm{ker}(E)$ such that $h=G^\dagger w$. It then follows that $\Lambda(G\mathcal{A}G^\dagger)w=G\Lambda(\mathcal{A})G^\dagger w=\mathbf{0}$. 
	
	Let $\kappa(z):=\det(zI_n-A)$.
	From $FG\mathcal{A}G^\dagger=AF$ and the Cayley-Hamilton theorem, $F\kappa(G\mathcal{A}G^\dagger)=\kappa(A)F=\mathbf{0}$.
	Therefore
	$\mathrm{im}(\kappa(G\mathcal{A}G^\dagger))\subseteq\mathrm{ker}(F)$,
	and hence $\pi(G\mathcal{A}G^\dagger)=\mathbf{0}$, where $\pi(z):=\Lambda(z)\kappa(z):=\sum_{i=0}^{2n}c_i z^i$, $c_{2n}=1$.
	To establish the full row rank property, suppose that
	$a\in\mathbb R^{n_v}$ and $b\in\mathbb R^m$ satisfy $a^\top V_0+b^\top U_0=\mathbf{0}$.
	Equivalently, $a^\top v_t+b^\top u_t=0$ for $t=0,\ldots,T-1$.
	Multiply the scalar equality at time $t+i$ by $c_i$, and sum over $i=0,\cdots,2n$; then
	$$0=\sum_{i=0}^{2n}c_i(a^\top v_{t+i}+b^\top u_{t+i}).$$
	Expanding the state recursion gives
	$$v_{t+i}=(G\mathcal{A}G^\dagger)^iv_t+\sum_{k=0}^{i-1}(G\mathcal{A}G^\dagger)^{i-1-k}G\mathcal{B}u_{t+k}.$$
	Substituting the state expansion gives
	\begin{align*}
		0=a^\top\left(\sum_{i=0}^{2n} c_i(G\mathcal{A}G^\dagger)^i\right)v_t+\sum_{i=1}^{2n}\sum_{k=0}^{i-1}c_ia^\top(G\mathcal{A}G^\dagger)^{i-1-k}G\mathcal{B} u_{t+k}+\sum_{i=0}^{2n}c_ib^\top u_{t+i}.
	\end{align*}
	Using $\pi(G\mathcal{A}G^\dagger)=\mathbf{0}$, we obtain
	\begin{align*}
		0=\sum_{k=0}^{2n}\left(a^\top \sum_{i=k+1}^{2n}c_i(G\mathcal{A}G^\dagger)^{i-1-k}G\mathcal{B} +c_kb^\top\right) u_{t+k}:=\sum_{k=0}^{2n}L_k u_{t+k},
	\end{align*}
	where $L_{2n}=b^\top$.
	Collecting these equations for $t=0,\cdots,T-2n-1$ yields
	$$[L_0,\cdots,L_{2n}]\mathcal{H}_{2n+1}(u_{[0,T-1]})=\mathbf{0}.$$
	The PE assumption implies $L_k=0$ for every $k$. In particular, $L_{2n}=0$ gives $b=\mathbf{0}$. The original relation consequently reduces to $a^\top V_0=\mathbf{0}$.
	The same PE assumption also implies that $u_{[0,T-2]}$ is PE of order $2n$. Combined with Lemma \ref{le4}, it holds that $V_0$ has full row rank.
	Thus, $a=\mathbf{0}$, proving (a).
    
For assertion (b), define $\Sigma_0:=\{(\mathcal{A}_0,\mathcal{B}_0):$
$\mathcal{A}_0Z_0+\mathcal{B}_0U_0=\mathbf{0}_{n_{\zeta}\times n_{\zeta}}\}$.
We now show that $u_t = K_{\zeta} \zeta_t$ cannot yield a result analogous to (\ref{PK}) when $Z_0$ is row-rank deficient.
In this case, there exists a nonzero vector
$a\in\mathrm{ker}(Z_0^\top)$.
Choose $\mathcal{A}_0=aa^\top$, $\mathcal{B}_0=\mathbf{0}$.
Then $(\mathcal{A}_0,\mathcal{B}_0)\in\Sigma_0$, whereas $[\mathcal{A}_0,\mathcal{B}_0][I_{n_\zeta}, K_\zeta^\top]^\top=aa^\top\neq\mathbf{0}$ for every $K_\zeta$.
Consequently, $\mathrm{ker}([Z_0^\top,U_0^\top])
\not\subseteq
\mathrm{ker}([I_{n_{\zeta}},K_{\zeta}^\top])$, i.e., $\mathrm{im}([I_{n_\zeta}, K_\zeta^\top]^\top)\not\subseteq\mathrm{im}
([Z_0^\top,U_0^\top]^\top)$.    
For the case where $u_t=Kv_t$, Lemma \ref{le4} gives $Z_0=G^\dagger V_0$.
Hence every
$(\mathcal A_0,\mathcal B_0)\in\Sigma_0$ satisfies $[\mathcal{A}_0G^\dagger,\mathcal{B}_0]
[V_0^\top,U_0^\top]^\top=\mathbf{0}$.
By assertion (a), the data matrix $[V_0^\top,U_0^\top]^\top$ has full
row rank, thus $\mathcal{A}_0G^\dagger=\mathbf{0}$ and $\mathcal{B}_0=\mathbf{0}$.
Therefore, for every $K$, $[\mathcal{A}_0,\mathcal{B}_0][(G^\dagger)^\top,K^\top]^\top=\mathbf{0}$,
$\forall(\mathcal{A}_0,\mathcal{B}_0)\in\Sigma_0$.
Equivalently,
$\mathrm{ker}([Z_0^\top,U_0^\top])
\subseteq
\mathrm{ker}([(G^\dagger)^\top,K^\top])$.
Taking orthogonal complements yields $\mathrm{im}
([(G^\dagger)^\top,K^\top]^\top)\subseteq\mathrm{im}([Z_0^\top,U_0^\top]^\top)$.
\end{proof}

We note that the output feedback controller $u_t = K v_t$ is essentially a DOF controller.
According to the discussion at the beginning of Section \ref{section3} on the relationship between the output feedback controller $u_t=Kv_t$ and the state feedback controller $u_t=K_x x_t$, we restrict attention to the admissible
gain set $\mathcal{U}_{ad}=\{K\in\mathbb{R}^{m\times n_v}\vert \exists K_x\in\mathbb{R}^{m\times n} \text{such that } \rho(A+BK_x)<1 \text{ and } K_xF=K\}$. 
This class contains the gain $K_x^*F$ induced by the optimal state-feedback controller. As shown in Section \ref{section4}, the proposed iterations can be related to their state-feedback counterparts.
Therefore, Problem \ref{P1} can be reduced to Problem \ref{P2} below.

\begin{problem}\label{P2}
	Consider system (\ref{system}) with unknown $A$, $B$, $C$ and unmeasurable $\{x_t\}$. Using the full-row-rank matrices $U_0$, $Y_0:=[y_0,\cdots,y_{T-1}]$, $[V_0^\top, U_0^\top]^\top$, and $V_1:=[v_1,\cdots,v_{T}]$, find an optimal output feedback gain $K^*$ such that $\{u^*_t=K^*v_t\}$ satisfies equation (\ref{pmin}).
\end{problem}

%%%%%%%%%%%%%%%%%%%%%%%%%%%%%%%%%%%%%%%%%%%%%%
\section{Efficient Off-Policy Value-Based Iteration Algorithms}\label{section4}
This section aims to develop efficient off-policy PI and VI algorithms for Problem \ref{P2}.  
The algorithm design focuses on using the full-row-rank matrix $[V_0^\top,U_0^\top]^\top$ to enhance the PI and VI iteration equations, thereby fully exploiting the system structure and input-output data.
The algorithm analysis focuses on elaborating the differences between output feedback and state feedback results. The last subsection details the advantages of the proposed algorithms.

Under the output feedback controller $u_t=Kv_t$, the cost functional (\ref{J}) can serve as the corresponding value function:
\begin{equation}\label{V}
	\begin{aligned}
		V^{K}(x_k):&=\sum_{t=k}^{\infty}\left(y_t^\top Qy_t+u_t^\top Ru_t\right)\vert_{u_t=Kv_t}\\
        &=y_k^\top Qy_k+v_k^\top K^\top R K v_k+V^{K}(Fv_{k+1}).
	\end{aligned}
\end{equation}
 Similar to the cost functional (\ref{J}), $V^{K}$ can also be expressed as the quadratic form $v_k^\top P v_k$ with $P=F^\top P^{K_x}F$, and substituting it into (\ref{V}) gives
$$\begin{aligned}
	&V^{K}(Fv_k)
	=v_k^\top F^\top Q_xFv_k+v_k^\top K^\top RKv_k+v_{k+1}^\top P v_{k+1}\\
	&=v_k^\top\begin{bmatrix}
		I\\K
\end{bmatrix}^\top\scalebox{0.9}{$\displaystyle\begin{bmatrix}
			F^\top(A^\top P^{K_x}A+Q_x)F&F^\top A^\top P^{K_x}B\\ B^\top P^{K_x}AF&B^\top P^{K_x}B+R
	\end{bmatrix}$}\begin{bmatrix}
		I\\K
	\end{bmatrix}v_k.
\end{aligned} $$

Define Q-function $Q^{K}(x_k,u_k)$ under $u_t=Kv_t$ as the value function which takes the input $u_k$ at time $k$ and applies feedback input $u_t=Kv_t$ from time $k+1$,
\begin{align}
\displaybreak[0]
	\nonumber &Q^{K}(x_k,u_k):=y_k^\top Qy_k+u_k^\top Ru_k+V^{K}(x_{k+1})\\\nonumber
    &=v_t^\top F^\top Q_xFv_t+u_k^\top Ru_k+v_{k+1}^\top P v_{k+1}\\\nonumber
	&=\begin{bmatrix}
		v_k\\u_k
	\end{bmatrix}^\top\begin{bmatrix}
		F^\top (A^\top P^{K_x}A+Q_x)F&F^\top A^\top P^{K_x}B\\ B^\top P^{K_x}AF&B^\top P^{K_x}B+R
	\end{bmatrix}\begin{bmatrix}
		v_k\\u_k
	\end{bmatrix}\\\label{Q}
	:&=\psi_k^\top\begin{bmatrix}
\Theta_{vv}&\Theta_{vu}\\\Theta_{vu}^\top&\Theta_{uu}
	\end{bmatrix}\psi_k:=\psi_k^\top \Theta \psi_k;
\end{align}
the final equation implies that the Q-function can also be expressed in quadratic form. Moreover, the following relationships hold
$$V^{K}(x_k)=Q^{K}(x_k,Kv_k),\quad P=[I_{n_v},K^\top]\Theta[I_{n_v},K^\top]^\top.$$
For a fixed policy $K$,
$\min_{u_t} Q^{K}(x_t,u_t)=\min_{u_t}\psi_t^\top\Theta \psi_t$ yields the improved policy $u_t^{+}=K^{+}v_t$, with
\begin{equation}\label{u*}
	K^{+}=-\Theta_{uu}^{-1}\Theta_{vu}^\top.
\end{equation}
Substituting equation (\ref{u*}) into $P=[I_{n_v},K^\top]\Theta[I_{n_v},K^\top]^\top$ yields the updated $P^{up}$ as
\begin{equation}\label{Pim}
	P^{up}=\Theta_{vv}-\Theta_{vu}\Theta_{uu}^{-1}\Theta_{vu}^\top.
\end{equation}
If both $P$ and $\Theta$ are treated as integral unknown parameters, we can utilize value function (\ref{V}), Q-function (\ref{Q}), policy improvement equation (\ref{u*}), and value update equation (\ref{Pim}) for model-free PI or VI.

\subsection{Q-Value Based Policy Iteration}\label{section4_1}
This subsection focuses on an off-policy PI algorithm. From equation (\ref{Q}), it holds that
\begin{equation}\label{PI0}
	\psi_k^\top\Theta\psi_k=y_k^\top Qy_k+u_k^\top Ru_k+v_{k+1}^\top\begin{bmatrix}
		I_{n_v}\\K
	\end{bmatrix}^\top\Theta\begin{bmatrix}
		I_{n_v}\\K
	\end{bmatrix} v_{k+1}.
\end{equation}
When the current control gain $K$ is known, $\Theta$ is the only unknown in equation (\ref{PI0}), and thus equation (\ref{PI0}) can be used to evaluate $Q^K$. Noting that the value of equation (\ref{PI0}) is a scalar, but to solve for the matrix $\Theta$, we need to construct a system of equations. 
To address this, under a $(2n+1)$-th order PE condition, we can use $U_0$, $Y_0$, $V_0$ and $V_1$ to form the following system of equations:
\begin{equation}\label{PI1}
	\Psi_0^\top \Theta\Psi_0=Y_0^\top Q Y_0+U_0^\top RU_0+V_1^\top\begin{bmatrix}
		I_{n_v}\\K
	\end{bmatrix}^\top \Theta \begin{bmatrix}
		I_{n_v}\\K
	\end{bmatrix}V_1,
\end{equation}
where $\Psi_0:=[\psi_0,\cdots,\psi_{T-1}]=[V_0^\top,U_0^\top]^\top$ has full row rank
with $T\geq n_v+m$. Equation (\ref{PI1}) is a generalized Sylvester equation and can be solved directly without converting it into an LS regression form.
Thus, equations (\ref{PI1}) and (\ref{u*}) can serve as the policy evaluation and policy improvement steps in PI, respectively. The complete procedure is given in Algorithm \ref{al1}. Since the data matrices are constructed only once before the iterations, using arbitrary input-output data that satisfy the PE condition, the algorithm is off-policy.

\begin{algorithm}
	\caption{Output-Feedback Policy Iteration (PI)}
	\begin{algorithmic}
		\State \textbf{Data Collection:} applying PE inputs of order $2n+1$ to system (\ref{system}), collect the corresponding output trajectories, and construct $U_0$ and $Y_0$. Generate either $H_0$ via equation (\ref{S0}) or $\Xi_0$. Using Lemma \ref{le4}, transform $H_0$ (or $\Xi_0$) into $V_0$ with $n_v$ rows, and construct $V_1$ using the same $G$. 
		\State \textbf{Initialization:} Set a precision constant $\epsilon$. Set $i=0$ and choose an initial stabilizing feedback gain $K^i\in\mathbb{R}^{m\times n_v}$.
		\For{$i=0,1,2,\cdots,$}
		\State (i) Policy evaluation: calculate $\Theta^{i+1}$ by solving
		\begin{equation}\label{PIi}
			\begin{split}
			    \Psi_0^\top \Theta^{i+1}\Psi_0=Y_0^\top Q Y_0+U_0^\top RU_0+{\small V_1^\top\begin{bmatrix}
					I_{n_v}\\K^i
				\end{bmatrix}^\top \Theta^{i+1} \begin{bmatrix}
					I_{n_v}\\K^i
				\end{bmatrix}V_1}.
			\end{split}
		\end{equation}
		\State (ii) Policy improvement: update $K^{i+1}$ by
		$$K^{i+1}=-(\Theta^{i+1}_{uu})^{-1}(\Theta^{i+1}_{vu})^\top.$$
		\If{$\Vert K^{i+1}-K^{i}\Vert\leq\epsilon$}
		\State~~~~\Return $K_{PI}^*=K^{i+1}$.
		\EndIf
		\EndFor
	\end{algorithmic}
	\label{al1}
\end{algorithm}

\begin{remark}[Initial Stabilizing Control Gain]
    Model-free construction of an initial stabilizing gain $K^0\in\mathcal{U}_{ad}$ remains an open problem when system matrices are unknown and states are unmeasured. Nevertheless, with suitable prior information, namely, a matrix $\mathcal{P}\in\mathbb{R}^{n\times n_v}$ satisfying $\mathrm{ker}(\mathcal{P})=\mathrm{ker}(F)$, a valid $K^0$ can be computed in a model-free manner. (This prior condition is inspired by the idea of recovering shared representations across multiple expert tasks; see \cite{3A_CDC}.) 
    In this case, there exists an invertible matrix
    $\mathcal{N}$ such that $\mathcal{P}=\mathcal{N} F$. Define $\tilde{x}_t:=\mathcal{P}v_t=\mathcal{N} x_t$, which obeys $\tilde{x}_{t+1}=\mathcal{N} A \mathcal{N}^\dagger\tilde{x}_t+\mathcal{N} Bu_t$. The full-row-rank matrices $\hat{V}_0:=\mathcal{P}V_0$ and $\hat{V}_1:=\mathcal{P}V_1$, together with $U_0$, satisfy the dynamics of this transformed system. Hence, the approaches in \cite{New-Per} can yield a right inverse $\tilde{K}_x$ of $\hat{V}_0$ such that $\hat{V}_1\tilde{K}_x$ is Schur stable, via either of the following schemes: (a) Let $g(v)$ be a basis matrix of $\mathrm{ker}(\hat{V}_0)$. Design $K^0_d$ such that $\hat{V}_1\hat{V}_0^\dagger+\hat{V}_1g(v)K^0_d$ is stable. Then, $\tilde{K}_x:=\hat{V}_0^\dagger+g(v)K_d^0$.
(b) If $\Sigma$ satisfies
	$\hat{V}_0\Sigma=(\hat{V}_0\Sigma)^\top$ and ${\small\begin{bmatrix}
			\hat{V}_0\Sigma&\hat{V}_1\Sigma\\(\hat{V}_1\Sigma)^\top&\hat{V}_0\Sigma
	\end{bmatrix}}\succ0$,
	then $\tilde{K}_x=\Sigma(\hat{V}_0\Sigma)^{-1}$. Define $K^0:=U_0\tilde{K}_x\mathcal{P}$. The gain $K_x^0:=U_0\tilde{K}_x\mathcal{N}$ satisfies $K^0=K^0_xF$, and $A+BK_x^0=\mathcal{N}^{-1}(\hat{V}_1\tilde{K}_x)\mathcal{N}$ is Schur stable. Therefore, $K^0$ can serve as the initial stabilizing control gain for Algorithm \ref{al1}.
\end{remark}

Define $\bar{F}:=\mathrm{diag}(F,I_m)$ and let $\Theta_{K_x}$ be the unique solution of the state-feedback Lyapunov equation (\ref{ly}). Using $X_0=FV_0$, $X_1=FV_1$, and $K=K_xF$, it holds that $\hat{\Theta}:=\bar{F}^\top\Theta_{K_x}\bar{F}$ satisfies equation (\ref{PI1}).
Since equation (\ref{PI1}) has a unique solution, $\Theta=\bar{F}^\top\Theta_{K_x}\bar{F}$. Combined with $X_1=AX_0+BU_0$,
equation (\ref{PI1}) can be transformed into
\begin{equation}\label{t1}
	\begin{split}
		\begin{bmatrix}
			X_0\\U_0
		\end{bmatrix}^\top\Theta_{K_x}\begin{bmatrix}
			X_0\\U_0
		\end{bmatrix}
		=\begin{bmatrix}
			X_0\\U_0
		\end{bmatrix}^\top\begin{bmatrix}
			Q_x&\\&R
		\end{bmatrix}\begin{bmatrix}
			X_0\\U_0
		\end{bmatrix}+\begin{bmatrix}
			X_0\\U_0
		\end{bmatrix}^\top[
		A,B]^\top\begin{bmatrix}
			I_n\\K_x
		\end{bmatrix}^\top\Theta_{K_x}\begin{bmatrix}
			I_n\\K_x
		\end{bmatrix}[A,B]
		\begin{bmatrix}
			X_0\\U_0
		\end{bmatrix}.
	\end{split}
\end{equation}
By the full-row-rank property of $[X_0^\top,U_0^\top]^\top$, the above equation is equivalent to the Lyapunov equation:
\begin{equation}\label{ly}
	\Theta_{K_x}
	=\mathrm{diag}(Q_x,R)+K_{AB}^\top\Theta_{K_x}K_{AB},
\end{equation}
where $K_{AB}:=[I_n,K_x^\top]^\top[A,B]$.
Thus, the iterative essence of Algorithm \ref{al1}, which is applicable to the case of unmeasurable states, is equivalent to state feedback PI in \cite{Efficient-Q}. If the states of system (\ref{system}) are known, Algorithm \ref{al1} can reduce to the latter.
Furthermore, the stability, convergence, and optimality of Algorithm \ref{al1} can be established similarly to state feedback PI results in \cite{Efficient-Q}, taking into account the effect of the full-row-rank matrix $F$ and the positive semi-definite $Q_x$.

\begin{lemma}\label{le6}
	Consider Algorithm \ref{al1}. If the initial feedback gain $K^0$ is stabilizing, then equation (\ref{PIi}) has a unique solution, all the feedback gains $K^i$ generated by Algorithm \ref{al1} are stabilizing, and the sequence $\{\Theta^i\}$ is monotonically non-increasing, eventually converging to $\Theta^*$ corresponding to the optimal feedback gain $K^*$. The gain
    $K_{PI}^{*}$ obtained via Algorithm \ref{al1} is an approximate optimal solution to Problem \ref{P2}.
\end{lemma}

\begin{proof}
The proof can follow the line of reasoning in \cite{Efficient-Q}; here we only highlight the parts that differ from it. 
First, we need to consider that the stability of $K^i$ is equivalent to that of $K_x^i$, 
and that $\Theta^{i}=\bar{F}^\top\Theta_{K_x}^{i}\bar{F}$. Second, note that $Q_x$ is only positive semi-definite, so the idea of \cite{Efficient-Q} cannot be directly applied when proving the stability of $K^{i+1}$. To address this, we need to use the observability of
$(A,Q_x^{\frac{1}{2}})$. Specifically, after we obtain
    $\Theta^{i+1}_{K_x}
	\geq \mathrm{diag}(Q_x,R)+(K_{AB}^{i+1})^\top\Theta^{i+1}_{K_x}K_{AB}^{i+1}$, where $K_{AB}^{i+1}$ has the same stability property as $A+BK_x^{i+1}$ \cite{Efficient-Q}, it suffices to prove that $K_{AB}^{i+1}$ is stabilizing. We proceed by contradiction. Assume there exists $\vert\lambda\vert\geq1$ and a nonzero vector $\mu=[\mu_x^\top,\mu_u^\top]^\top$ such that $K_{AB}^{i+1}\mu=\lambda\mu$. Premultiplying and postmultiplying the Lyapunov inequality by $\mu^\top$ and $\mu$, respectively, gives
    $(1-\vert\lambda\vert^2)\mu^\top \Theta_{K_x}^{i+1}\mu\geq \mu_x^\top Q_x\mu_x+\mu_u^\top R\mu_u\geq0$. Since $\Theta_{K_x}^{i+1}\in\mathbb{S}_+^{m+n}$ and $R\in\mathbb{S}_{++}^m$, it follows that
    $\mu_x^\top Q_x\mu_x=0$ and $\mu_u=\mathbf{0}$. Thus, $A\mu_x=\lambda\mu_x$ and $Q_x^{\frac{1}{2}}\mu_x=\mathbf{0}$, contradicting the observability of $(A,Q_x^{\frac{1}{2}})$. Therefore, $K_{AB}^{i+1}$ is stabilizing, and so is $K^{i+1}$. The rest of the proof carries over analogously to \cite{Efficient-Q}.
    \end{proof}

\begin{lemma}[Second-Order Convergence Rate]\label{le7}
	Algorithm \ref{al1} has a second-order convergence rate, that is,
	    $$\Vert \Theta^{i+1}-\Theta^*\Vert\leq\beta_1\Vert\Theta^{i}-\Theta^*\Vert^2,$$
	where bounded constant $\beta_1:=\sup_{i}\sum_{j=1}^{\infty}\Vert(K_{AB}^{i})^j\Vert^2$
    $\cdot\sigma_{\min}^{-1}(B^\top P^{K_x^*}B+R)\small\frac{\sigma^2_{\max}(\bar{F})}{\sigma^2_{\min}(\bar{F})}$. Choose positive integer $i_0$ such that $\beta_1\Vert\Theta^{i_0}-\Theta^*\Vert<1$. If $0<\epsilon<\Vert\Theta^{i_0}-\Theta^*\Vert$, then $\Vert\Theta^{i_0+k}-\Theta^*\Vert\leq \epsilon$, whenever
	 $k\geq\log_2\left(\frac{\log(\beta_1\epsilon)}{\log(\beta_1\Vert\Theta^{i_0}-\Theta^{*}\Vert)}\right)$. 
\end{lemma}
\begin{proof}
	Define $\bar{\Theta}_{K_x}^{*}:=\mathrm{diag}(\mathbf{0}_{n\times n},(\Theta_{K_x,uu}^{*})^{-1})$, where $\Theta_{K_x,uu}^{*}$ denotes the input block of $\Theta_{K_x}^*$. Based on the equivalence established above for the controller and Lemma \ref{le6}, equation $\Theta^{i+1}_{K_x}-\Theta^{*}_{K_x}=\sum_{j=1}^{\infty}((K_{AB}^{i})^j)^\top(\Theta^{i}_{K_x}-\Theta^{*}_{K_x})\bar{\Theta}_{K_x}^{*}(\Theta^{i}_{K_x}-\Theta^{*}_{K_x})(K_{AB}^{i})^j$ in \cite{Efficient-Q} still holds. Since $\Theta_{K_x,uu}^{*}=B^\top P^{K_x^{*}}B+R=\Theta_{uu}^{*}$, it follows that $\bar{\Theta}_{K_x}^{*}=\mathrm{diag}(\mathbf{0}_{n\times n},(\Theta_{uu}^{*})^{-1})$. 
	Combined with the full-row-rank property of $\bar{F}$, the following equation under output feedback can be obtained.
	\begin{equation}
		\nonumber\begin{split}
			\Theta^{i+1}-\Theta^{*}=
&\sum_{j=1}^{\infty}\bar{F}^\top((K_{AB}^{i})^j)^\top(\bar{F}^\dagger)^\top(\Theta^{i}-\Theta^{*})\\
&\cdot\mathrm{diag}(\mathbf{0}_{n_v\times n_v},(\Theta_{uu}^{*})^{-1})(\Theta^{i}-\Theta^{*})\bar{F}^\dagger(K_{AB}^{i})^j\bar{F}.
		\end{split}
	\end{equation}
	Taking the norm on both sides of the above equation yields
		$$\Vert\Theta^{i+1}-\Theta^{*}\Vert\leq \scalebox{0.93}{$\displaystyle  \sum_{j=1}^{\infty}\Vert(K_{AB}^{i})^j\Vert^2\Vert(\Theta_{uu}^{*})^{-1}\Vert\frac{\sigma^2_{\max}(\bar{F})}{\sigma^2_{\min}(\bar{F})}\Vert\Theta^{i}-\Theta^{*}\Vert^2$}.$$
   By Lemma \ref{le6}, $K_{AB}^i$ is Schur stable and $K_{AB}^i\to K_{AB}^*$, with $K_{AB}^*$ also Schur stable. Let $W_i$ solve $W_i=I+(K_{AB}^i)^\top W_iK_{AB}^i$. Continuity of the Lyapunov solution gives $W_i\to W_*$, and $\sum_{j=1}^{\infty}\|(K_{AB}^i)^j\|^2\le\mathrm{trace}(W_i-I)$. Hence, $\sum_{j=1}^{\infty}\Vert(K_{AB}^{i})^j\Vert^2$ is uniformly bounded.
   The full-row-rank property of $F$ implies that $\sigma_{\max}(\bar{F})/\sigma_{\min}(\bar{F})$ is bounded. Combined with $\Vert(\Theta_{uu}^*)^{-1}\Vert=\sigma_{\min}^{-1}( B^\top P^{K_x^*}B+R)$, $\beta_1$ is a bounded constant. Since $\Vert\Theta^{i}-\Theta^*\Vert\to 0$, there exists an index $i_0$ satisfying $\beta_1\Vert\Theta^{i_0}-\Theta^*\Vert<1$. Hence, $\beta_1\Vert\Theta^{i_0+k}-\Theta^*\Vert\leq(\beta_1\Vert\Theta^{i_0}-\Theta^*\Vert)^{2^k}$. The stated iteration bound follows by solving $(\beta_1\Vert\Theta^{i_0}-\Theta^*\Vert)^{2^k}\leq \beta_1\epsilon.$
\end{proof}

\subsection{Q-Value Based Value Iteration}\label{section4_2}
This subsection focuses on an off-policy VI algorithm that does not require an initial stabilizing controller. Given the parameter matrix $P$ of the current value function, the following equation holds, in which the only unknown matrix is $\Theta$ corresponding to the Q-function.
\begin{equation}\label{VI0}
	\nonumber\psi_k^\top \Theta\psi_k=y_k^\top Q y_k+u_k^\top R u_k+v_{k+1}^\top Pv_{k+1}.
\end{equation}
Similar to the perspective of considering the system of equations in Subsection \ref{section4_1}, here we focus on
\begin{equation}\label{VI1}
	\Psi_0^\top \Theta \Psi_0=Y_0^\top QY_0+U_0^\top RU_0+V_1^\top PV_1.
\end{equation}
Due to the full-row-rank property of $\Psi_0$, we can directly compute $\Theta$ from equation (\ref{VI1}). Note that under the improved policy (\ref{u*}), the updated $P^{up}$ is given by equation (\ref{Pim}). Thus, equations (\ref{VI1}) and (\ref{Pim}) can serve as the Q-value evaluation and value update steps in VI, respectively. If $P$ has converged to the fixed point, the corresponding $\Theta$ can be used to obtain the feedback gain via equation (\ref{u*}). The complete procedure is given in Algorithm \ref{al2}. This algorithm is also off-policy.

\begin{algorithm}[h]
	\caption{Output-Feedback Value Iteration (VI)}
	\begin{algorithmic}
		\State \textbf{Data Collection:} Applying PE inputs of order $2n+1$ to system (\ref{system}), collect the corresponding output trajectories,  and construct $U_0$ and $Y_0$. Generate either $H_0$ via equation (\ref{S0}) or $\Xi_0$. Using Lemma \ref{le4}, transform $H_0$ (or $\Xi_0$) into $V_0$ with $n_v$ rows, and construct $V_1$ using the same $G$. 
		\State \textbf{Initialization:} Set a precision constant $\epsilon$. Set $i=0$ and choose an initial matrix $P^i\in\mathbb{S}_{+}^{n_v}$.
		\For{$i=0,1,2,\cdots,$}
		\State (i) Q-value evaluation: calculate $\Theta^{i}$ by
		\begin{equation}\label{VIi}
			\Psi_0^\top \Theta^{i}\Psi_0=Y_0^\top Q Y_0+U_0^\top RU_0+V_1^\top P^iV_1.
		\end{equation}
		\State (ii) Value update: update $P^{i+1}$ by
		$$P^{i+1}= \Theta^{i}_{vv}-\Theta^{i}_{vu}(\Theta_{uu}^i)^{-1}(\Theta^{i}_{vu})^\top.$$
		\If{$\Vert P^{i+1}-P^{i}\Vert\leq\epsilon$}
		\State calculate $\Theta^{i+1}$ by equation (\ref{VIi}).
		\State~~~~\Return $K_{VI}^*=-(\Theta_{uu}^{i+1})^{-1}(\Theta^{i+1}_{vu})^\top$.
		\EndIf
		\EndFor
	\end{algorithmic}
	\label{al2}
\end{algorithm}

Let $A_v:=G\mathcal{A}G^\dagger$, $B_v:=G\mathcal{B}$ and $Q_v:=F^\top Q_xF$.
Since $V_1=[A_v,B_v]\Psi_0$ and $Y_0=CFV_0$, the full row rank property of $\Psi_0$ implies that (\ref{VIi}) is equivalent to $\Theta^i=\mathrm{diag}(Q_v,R)+[A_v,B_v]^\top P^i[A_v,B_v]$.
Consequently, Algorithm \ref{al2} implements the Riccati
iteration for the consistent system with $v_{t}$ as the state variable, i.e.,
\begin{equation}\label{VIv}
    \begin{split}
        P^{i+1}
=Q_v+A_v^\top P^iA_v-A_v^\top P^iB_v
(R+B_v^\top P^iB_v)^{-1}
B_v^\top P^iA_v.
    \end{split}
\end{equation}

\begin{lemma}[Linear Convergence Rate]\label{le8}
Algorithm \ref{al2} converges to $P^*:=F^\top P^{K_x^*}F$ and $K^*:=K_x^*F$. Moreover, there exists a positive integer $i_0$ such that
$$\Vert P^i-P^*\Vert\leq \beta_2 q^{i-i_0}\Vert P^{i_0}-P^*\Vert,$$
where $\beta_2\geq 1$ and $\rho(A_v+B_vK^*)^2<q<1$.
\end{lemma}

\begin{proof}
Note that $FA_v=AF$, $FB_v=B$ and $\mathrm{ker}(F)$ is $A_v$-invariant.
	We now prove that for every $w\in\mathrm{ker}(F)$, it holds that $\lim_{k\to\infty}A_v^kw=\mathbf{0}$.
    For the history trajectory-based state parameterization, if $w\in\mathrm{ker}(F)$, then $x=Fw=\mathbf{0}$. Starting from this zero state and with zero future inputs, both the state and output remain zero for all subsequent time steps. This yields $A_v^Nw=\mathbf{0}$, so the restricted dynamics are nilpotent.
    For the observer-based state parameterization, recall that $\eta_t=G^\dagger v_t$. For $w\in\mathrm{ker}(F)$, we have $x=S\eta=SG^\dagger w=Fw=\mathbf{0}$. Starting from this zero state and with zero future inputs, the output stays identically zero, and the dynamics for $\eta_t$ reduce to $\eta_{t+1}=\mathcal{A}_s\eta_t$. Since all eigenvalues of $\mathcal{A}_s$ lie inside the unit circle, $\mathcal{A}_s$ is Schur stable. This establishes that $\lim_{k\to\infty}A_v^kw=\mathbf{0}$ for every $w\in\mathrm{ker}(F)$. 
    Choose a gain $K_x$ such that $A+BK_x$ is Schur stable. 
    Since $F(A_v+B_vK_xF)=(A+BK_x)F$, any nonzero vector $h$ satisfying $(A_v+B_vK_xF)h=\lambda h$ and $|\lambda|\geq1$ must satisfy $Fh=\mathbf{0}$. Then $A_vh=\lambda h$, contradicting
    $A_v^kh\to\mathbf{0}$ for every $h\in\mathrm{ker}(F)$. Thus, $A_v+B_vK_xF$ is Schur stable, and
    $(A_v,B_v)$ is stabilizable. Similarly, suppose that $A_v\phi=\mu \phi$ and
    $Q_x^{1/2}F\phi=\mathbf{0}$ for some nonzero $\phi$ and $|\mu|\geq1$.
    Since $A(F\phi)=F(A_v\phi)=\mu (F\phi)$, detectability of $(A,Q_x^{1/2})$ implies $F\phi=\mathbf{0}$.
    This again contradicts $A_v^k\phi\to\mathbf{0}$. Hence, $(A_v,Q_x^{1/2}F)$ is detectable.
	Therefore, the Riccati iteration (\ref{VIv}) converges from every $P^0\in\mathbb S_+^{n_v}$ to the unique stabilizing solution $P^*$ \cite{VI-noises}. Using $FA_v=AF$ and $FB_v=B$, substitution shows that $F^\top P^{K_x^*}F$ satisfies the DARE associated with $A_v$ and $B_v$, and yields the gain $K_x^*F$. The preceding stability argument also applies to $A_v+B_vK_x^*F$.
	Hence, uniqueness of the stabilizing solution gives $P^*=F^\top P^{K_x^*}F$ and $K^*=K_x^*F$.

	To establish the convergence rate, set $A_*:=A_v+B_vK^*$ and let $\mathscr{R}$ denote the Riccati map. Its derivative at $P^*$ is $D\mathscr{R}(P^*)[\Delta]=A_*^\top\Delta A_*$, as shown in the derivation below.
	Let $K_P:=-(R+B_v^\top PB_v)^{-1}B_v^\top PA_v$, and define the associated closed-loop matrix $A_P:=A_v+B_vK_P$. The Riccati map can then be written as
	$$\mathscr{R}(P)=Q_v+K_P^\top RK_P+A_P^\top PA_P.$$
	Perturb $P$ along the direction $\Delta$ and denote the first-order variation of the gain $K_P$ by $\delta K$. It follows that $\delta A_P=B_v\delta K$. Differentiating the Riccati map yields
	$$D\mathscr{R}(P)[\Delta]=A_P^\top \Delta A_P+\delta K^\top(RK_P+B_v^\top PA_P)+(RK_P+B_v^\top PA_P)^\top \delta K.$$
	From the definition of $K_P$, we have $RK_P+B_v^\top PA_P=\mathbf{0}$. Hence $D\mathscr{R}(P)[\Delta]=A_P^\top \Delta A_P$, which also holds at $P^*$.
	Since $A_*$ is Schur stable, pick any $q\in(\rho(A_*)^2,1)$. One may choose $\gamma$ satisfying $\rho(A_*)<\gamma<\sqrt{q}$. Define $\mathcal{N}:=\sum_{k=0}^{\infty}\left(\frac{A_{*}}{\gamma}\right)^{k,\top}\left(\frac{A_{*}}{\gamma}\right)^k\succ\mathbf{0}$, which satisfies $A_*^\top \mathcal{N}A_*=\gamma^2(\mathcal{N}-I_{n_v})\preceq\gamma^2 \mathcal{N}$. For symmetric matrices, introduce the weighted norm $\Vert \Delta\Vert_\mathcal{N}:=\Vert \mathcal{N}^{-1/2}\Delta \mathcal{N}^{-1/2}\Vert$. Then $\Vert \mathcal{N}^{1/2}A_*\mathcal{N}^{-1/2}\Vert\leq\gamma$. Combining this with the derivative of the Riccati map, we have
	$$\Vert D\mathscr{R}(P^*)[\Delta]\Vert_\mathcal{N}=\Vert \mathcal{N}^{-1/2}A_*^\top\Delta A_*\mathcal{N}^{-1/2}\Vert\leq\gamma^2\Vert \Delta\Vert_\mathcal{N}\leq q\Vert \Delta\Vert_\mathcal{N}.$$
	The derivative of the Riccati map is continuous near $P^*$, so there exists a neighborhood of $P^*$ such that the operator norm of the derivative is bounded by $q$ within this neighborhood. Accordingly,
	$$\Vert \mathscr{R}(P)-\mathscr{R}(P^*)\Vert_\mathcal{N}\leq q\Vert P-P^*\Vert_\mathcal{N}$$
	holds for all $P$ in the neighborhood.
	As $P^i\rightarrow P^*$, there exists a positive integer $i_0$ such that all subsequent iterates enter this neighborhood. Repeated application of the above inequality gives
	$$\Vert P^i-P^*\Vert_\mathcal{N}\leq q^{i-i_0}\Vert P^{i_0}-P^*\Vert_\mathcal{N},\quad i\geq i_0.$$
	Finally, using $\Vert \Delta\Vert\leq\lambda_{\max}(\mathcal{N})\Vert\Delta\Vert_\mathcal{N}$ and $\Vert\Delta\Vert_\mathcal{N}\leq\frac{1}{\lambda_{\min}(\mathcal{N})}\Vert\Delta\Vert$, where $\lambda_{\max}(\cdot)$ and $\lambda_{\min}(\cdot)$ denote the maximum and minimum eigenvalues, respectively, we obtain
	$$\Vert P^i-P^*\Vert\leq \beta_2q^{i-i_0}\Vert P^{i_0}-P^*\Vert,\quad i\geq i_0,$$
	with $\beta_2=\lambda_{\max}(\mathcal{N})/\lambda_{\min}(\mathcal{N})$.
\end{proof}

\subsection{Why Algorithms are Data-Efficient?}\label{section4_3}
This subsection demonstrates the efficiency of the proposed algorithms.
Prior works such as \cite{DT-output,Deng-DT-output} did not exploit full-row-rank data matrices. Instead, they directly perform PI or VI using the data sequence $\{\zeta_t\}$. In numerical implementation, they typically rely on the identity $s^\top X r=(r\otimes s)^\top \mathrm{vec}(X)$ to convert the Bellman iterative equations into an LS regression form. Unique recovery of all unconstrained quadratic
parameters requires the regression
matrix (comprising data matrices associated with $\mathrm{vech}(\zeta_{t}\zeta_t^\top)$, $\mathrm{vech}(\zeta_{t+1}\zeta_{t+1}^\top)$, $(u_{t}\otimes\zeta_{t})^\top$ and $\mathrm{vech}(u_{t}u_{t}^\top)$) to have full row rank; this is an indispensable convergence condition in many classical works \cite{DT-output,Deng-DT-output}. However, this condition is not easily satisfied, especially in MO problems where it is nearly impossible, as explained below.
From Lemma \ref{le4}, $\zeta_t=G^\dagger v_t$. Then $\mathrm{vech}(\zeta_{t}\zeta_t^\top)=\mathrm{vech}(G^\dagger v_tv_t^\top (G^\dagger)^\top)=\Pi_{n_\zeta}^\dagger(G^\dagger\otimes G^\dagger)\Pi_{n_v}\mathrm{vech}(v_t v_t^\top)$ holds, where $\Pi_{n_\zeta}$ and $\Pi_{n_v}$ are duplication matrices for the corresponding dimensions. Consequently, $\mathrm{rank}([\mathrm{vech}(\zeta_{0}\zeta_0^\top),\cdots,\mathrm{vech}(\zeta_{T-1}\zeta_{T-1}^\top)])\leq\mathrm{rank}([\mathrm{vech}(v_0 v_0^\top), \mathrm{vech}(v_1 v_1^\top),\cdots,\mathrm{vech}(v_{T-1} v_{T-1}^\top)])\leq n_v(n_v+1)/2<n_{\zeta}(n_{\zeta}+1)/2$. 

Thus, the first advantage of Algorithms \ref{al1} and \ref{al2} is that they do not require the data to satisfy additional, hard-to-satisfy conditions on regression matrices; the proposed algorithms are feasible for both SO and MO problems. The data matrix $\Phi_0$ must have full row rank. Moreover, both the generalized Sylvester equation (\ref{PIi}) and equation (\ref{VIi}) can be solved directly.

Second, Algorithms \ref{al1} and \ref{al2} require fewer unknowns and less data.
Due to the projection operation, the dimension of the substitute state reduces from $n_{\zeta}$ to $n_v$. Consequently, the number of unknowns in $\Theta$ decreases from $(n_{\zeta}+m)(n_{\zeta}+m+1)/2$ to $(n_v+m)(n_v+m+1)/2$, where $n_{\zeta}-n_v=(p-1)n$. In terms of data requirements, LS-based algorithms require the number of equations to exceed the number of unknowns, i.e., $T\geq (n_{\zeta}+m)(n_{\zeta}+m+1)/2$. In contrast, the proposed algorithms only require that the number of columns in $\Psi_0$ be no less than the number of its rows, i.e., $T\geq n_v+m$.

Third, Algorithms \ref{al1} and \ref{al2} have lower theoretical computational complexity per iteration. Excluding one-time preprocessing, Algorithm \ref{al1} solves the Sylvester equation (\ref{PIi}) with $O((n_v+m)^3)$ operations; Algorithm \ref{al2} solves equation (\ref{VIi}) using the pseudo-inverse of a full-row-rank data matrix, which also requires at most $O((n_v+m)^3)$ operations. In contrast, for LS-based PI or VI algorithms such as \cite{Deng-DT-output}, solving a dense
regression problem through normal equations requires $O(Ts^2+s^3)$ operations, where
$s=(n_\zeta+m)(n_\zeta+m+1)/2$.

%%%%%%%%%%%%%%%%%%%%%%%%%%%%%%%%%%%%

\section{Robustness Analysis and Further Discussions}\label{section5}
This section presents the robustness of the proposed algorithms, along with further discussions.

\subsection{Robustness Analysis}\label{section5_1}
This subsection considers inexact data corrupted by unmeasurable, small process noises $\{w_t\}$ and measurement noises $\{e_t\}$. The system dynamics are
\begin{equation}\label{noise}
	x_{t+1}=Ax_t+Bu_t+w_t,\quad y_{t}=Cx_t+e_t.
\end{equation}
In this case, the state parameterization equations (\ref{x_t}) and (\ref{x_t4}) become $x_t=M\xi_t+\hat{M}[w_{[t-n,t-1]}^\top,e_{[t-n,t-1]}^\top]^\top$ and $x_t=[S_u,S_y,\hat{S}_{\epsilon}][(\eta_t^u)^\top,(\eta_t^y)^\top,(\hat{\eta}_t^{\epsilon})^\top]^\top$, where $\hat{M}$ has a structure similar to $M$, $\hat{S}_{\epsilon}$ and $\hat{\eta}_t^{\epsilon}$ are required to replicate the observation error system (\ref{eps}) with an extra term $w_t-Le_t$ added to its right-hand side. The state parameterization matrices in this form still have full row rank. However, because the noises $\{w_t\}$, $\{e_t\}$ and $\epsilon_0\neq\mathbf{0}$ are unmeasurable, we still use $\xi_t$ and $\eta_t$ from (\ref{x_t}) and (\ref{x_t4}) as the substitute states. 
Define $W^{\mathcal{H}}_{0}:=\mathcal{H}_{N}(w_{[-N,T-2]})$ and $E_0^{\mathcal{H}}:=\mathcal{H}_{N}(e_{[-N,T-2]})$. Denote the matrices with the superscript $^{ex}$ as the corresponding noise-free counterparts. Then $\Delta_x:=X_0-EZ_0$ satisfies $\Vert\Delta_x\Vert\leq c_{xw}\Vert W^{\mathcal{H}}_{0}\Vert+c_{xe}\Vert E_0^{\mathcal{H}}\Vert$. 
From the difference between the noisy output $y_t$ and its exact value,
it is readily verifiable that the perturbation upper bound of the substitute state data matrix satisfies $\Vert Z_0-Z_0^{ex}\Vert\leq c_{zw}\Vert W_0^{\mathcal{H}}\Vert+c_{ze}\Vert E_0^{\mathcal{H}}\Vert$,
where the constants are finite and depend on the
system, the chosen parameterization, and the data
interval.

Noise may increase the rank of $Z_0$, yet the additional singular values can be very small, as they characterize the noise contribution. To mitigate the impact of noise on the data matrix, we can also employ the rank-$n_v$ truncated SVD method proposed after Lemma \ref{le4} to construct a full-row-rank matrix $V_0$ with row rank $n_v$. Moreover, when repeated experiments with the same inputs and initial conditions are available, multiple data matrices can be constructed and averaged to replace the original one, thereby mitigating noise to some extent.

Define $\Delta_0:=V_0-V_0^{ex}$, $\Delta_1:=V_1-V_1^{ex}$, with $\Vert \Delta_0\Vert$ and $\Vert\Delta_1\Vert$ admitting bounds similar to those of $\Vert Z_0-Z_0^{ex}\Vert$. Let $E_0:=[e_0,\cdots,e_{T-1}]$, $\Delta_{y}:=Y_0-CFV_0=C(X_0-FV_0) +E_0$, and $\Delta_{xy}:=\Delta_{y}^\top Q CFV_0+(CFV_0)^\top Q\Delta_{y}+\Delta_{y}^\top Q\Delta_{y}$.
The robustness of Algorithm \ref{al1} is now analyzed. Let $\hat{\Psi}=[(FV_0)^\top,U_0^\top]^\top$, $\bar{K}_x=[I_n,K_x^\top]^\top$ and $\Delta_A=F\Delta_1-AF\Delta_0$. Then $FV_1=[A,B]\hat{\Psi}+\Delta_A$.
Thus, the state-dependent iterative formula derived from equation (\ref{PI1}) is $\hat{\Psi}^\top \Theta_{K_x}\hat{\Psi}=\hat{\Psi}^\top(\mathrm{diag}(Q_x,R)+[A,B]^\top\bar{K}_x^\top\Theta_{K_x}\bar{K}_x[A,B])\hat{\Psi}+\Delta_{\epsilon}$, where
\begin{align*}
\Delta_\epsilon:=&\hat{\Psi}^\top[A,B]^\top\bar{K}_x^\top\Theta_{K_x}\bar{K}_x\Delta_A+\Delta_A^\top\bar{K}_x^\top\Theta_{K_x}\bar{K}_x[A,B]\hat{\Psi}\\
&+\Delta_A^\top\bar{K}_x^\top\Theta_{K_x}\bar{K}_x\Delta_A+\Delta_{xy}.
\end{align*}
 Therefore, $(\hat{\Psi}^\dagger)^\top\Delta_{\epsilon}\hat{\Psi}^\dagger$ corresponds to a perturbation added to $\mathrm{diag}(Q_x,R)$. 
Define $\Theta_{K_x}^{ex}$ as the state-related matrix obtained under exact data.
Then $\Theta_{K_x}-\Theta_{K_x}^{ex}=\sum_{j=0}^{\infty}(\bar{K}_x[A,B])^{j\top}(\hat{\Psi}^\dagger)^\top\Delta_{\epsilon}\hat{\Psi}^\dagger(\bar{K}_x[A,B])^{j}$. Define $c_{K_x}:=\sum_{j=0}^{\infty}\Vert (\bar{K}_x[A,B])^{j} \Vert^2<\infty$. Hence, $\Theta_{K_x}$ is positive definite if $\Vert \Delta_{\epsilon}\Vert<\sigma_{\min}(\Theta_{K_x}^{ex})\sigma_{\min}^2(\hat{\Psi})/c_{K_x}$. Define $K_x^{+,ex}:=-(\Theta_{K_x,uu}^{ex})^{-1}(\Theta_{K_x,xu}^{ex})^\top$ and $K_x^{+}:=-\Theta_{K_x,uu}^{-1}\Theta_{K_x,xu}^\top$. Then
$\Vert K_x-K_x^{ex}\Vert\leq c_{K_x}\Vert (\hat{\Psi}^\dagger)^\top\Delta_{\epsilon}\hat{\Psi}^\dagger\Vert(1+\Vert K_x^{+,ex}\Vert)/(\sigma_{\min}(R)-c_{K_x}\Vert (\hat{\Psi}^\dagger)^\top\Delta_{\epsilon}\hat{\Psi}^\dagger\Vert)$ holds if $c_{K_x}\Vert (\hat{\Psi}^\dagger)^\top\Delta_{\epsilon}\hat{\Psi}^\dagger\Vert<\sigma_{\min}(R)$.
Since the exact improved gain $K_x^{+,ex}$ is stabilizing, sufficiently small evaluation errors also preserve the stability of $A+BK_x^+$. 
Therefore, in conjunction with Lemma \ref{le6} and continuity, the following conclusion can be drawn.

\begin{lemma}
	Consider system (\ref{noise}) with small process noises $\{w_t\}$ and measurement noises $\{e_t\}$. Suppose Algorithm \ref{al1} maintains $\Theta^i=\bar{F}^\top\Theta_{K_x}^i\bar{F}$ and $K^i=K_x^iF$.
   There exist constants $\bar{\delta}>0$ and $c>0$ such that, if $\sup_{i\geq 0}\Vert(\hat{\Psi}^\dagger)^\top\Delta_{\epsilon}\hat{\Psi}^\dagger\Vert\leq\delta\leq\bar{\delta}$, then all gains $K^i$ generated by Algorithm \ref{al1}
are stabilizing and $\limsup_{i\rightarrow\infty}\Vert K^i-K_x^*F\Vert
\leq c\Vert F\Vert\delta$. If $(\hat{\Psi}^\dagger)^\top\Delta_{\epsilon}\hat{\Psi}^\dagger\rightarrow \mathbf{0}$, then Algorithm \ref{al1} converges to the optimal
gain $K_x^*F$.
\end{lemma}

We next analyze the robustness of Algorithm \ref{al2}. Similarly, define $A_v:=G\mathcal{A}G^\dagger$, $B_v:=G\mathcal{B}$, $Q_v:=F^\top Q_xF$ and $\Delta_B:=\Delta_1-A_v\Delta_0$. For a given
matrix $P$, the noise-free Q-value evaluation is $\Theta^{ex}(P)=\mathrm{diag}(Q_v,R)+[A_v,B_v]^\top P [A_v,B_v]$; using $V_1=[A_v,B_v]\Psi_0+\Delta_B$, the noisy evaluation satisfies $\Theta(P)=\Theta^{ex}(P)+(\Psi_0^\dagger)^\top\Delta_e\Psi_0^\dagger$, where
\begin{align*}
    \Delta_e:=&\Psi_0^\top [A_v,B_v]^\top P\Delta_B+\Delta_B^\top P[A_v,B_v]\Psi_0\\
    &+\Delta_B^\top P\Delta_B+\Delta_{xy}.
\end{align*}
Let $\mathcal{S}(\Theta):=\Theta_{vv}-\Theta_{vu}\Theta_{uu}^{-1}\Theta_{vu}^\top$, $\mathscr{R}(P)=\mathcal{S}(\Theta^{ex}(P))$ and $\Delta P^i=\mathcal{S}(\Theta(P^i))-\mathscr{R}(P^i)$. Then Algorithm \ref{al2} satisfies $P^{i+1}=\mathscr{R}(P^i)+\Delta P^i$. Since $\Theta_{uu}^{ex}\geq\sigma_{\min}(R)I_{m}$, sufficiently small perturbations keep $\Theta_{uu}$ uniformly positive definite. Thus, the Schur-complement update is locally Lipschitz continuous.
Hence, $\Vert \Delta P^i\Vert\leq c_P\Vert\Theta(P^i)-\Theta^{ex}(P^i)\Vert\leq c_P\Vert \Psi_0^\dagger\Vert^2\Vert\Delta_e\Vert$
with a bounded constant $c_P$. Combined with the proof of Lemma \ref{le8}, for some $q\in(0,1)$ and $\beta\geq1$, $\Vert P^{i+1}-P^{*}\Vert_{\mathcal{N}}\leq  q\Vert P^i-P^*\Vert_{\mathcal{N}}+\Vert \Delta P^i\Vert_{\mathcal{N}}$ and $\Vert P^i-P^*\Vert\leq\beta q^{i-i_0}\Vert P^{i_0}-P^*\Vert+\beta\sum_{j=i_0}^{i-1}q^{i-1-j}\Vert\Delta P^j\Vert$ 
holds within a sufficiently small neighborhood of $P^*$. This gives the following result.

\begin{lemma}
	Consider system (\ref{noise}) with small process noises $\{w_t\}$ and measurement noises $\{e_t\}$. There exist constants $r>0$, $\beta\geq1$, and $q\in(0,1)$ such that the following holds: if the iterates of Algorithm \ref{al2} satisfy
    $\Vert P^{i_0}-P^*\Vert \leq r/\beta$ and $\sup_{i\geq i_0}\Vert \Delta P^i\Vert<\delta<(1-q)r/\beta$, then $\Vert P^i-P^*\Vert\leq r$ for all $i\geq i_0$, and $\limsup_{i\to\infty}\Vert P^i-P^*\Vert\leq\frac{\beta\delta}{1-q}$. Furthermore, if $\Delta P^i\to\mathbf{0}$, then Algorithm \ref{al2} satisfies $P^i\to P^*$.
\end{lemma}
This subsection essentially requires the noise to be sufficiently small, and the derived upper bound condition is conservative. To achieve better robustness, compensation mechanisms need to be developed, which is left for future research.

\subsection{Further Discussions}\label{section5_2}
The proposed framework also offers the following notable flexibility and strong expandability.

(i) The PE data required in this paper need not be collected continuously. If the mosaic Hankel matrix formed from multiple short trajectories satisfies the corresponding full-row-rank property, i.e., collective persistency of excitation \cite{Ivan-data-driven}, it also contains all the information of system (\ref{system}) and can be used in the proposed algorithms. This advantage is particularly useful for scenarios with data loss or unstable systems that only permit short-duration experiments.

(ii) In fact, even if system (\ref{system}) is only stabilizable (not controllable), Algorithm \ref{al1} and Algorithm \ref{al2} may still converge to the optimal solution as long as the full-row-rank data matrix $\Psi_0$ is obtainable; this will be verified in Example \ref{ex3}. The controllability assumption on system (\ref{system}) in Section \ref{section3} is primarily for the convenience of theoretical analysis.

(iii) State parameterization (\ref{x_t4}) can be extended to continuous-time problems by replacing the z-transform with the Laplace transform. Thus, we expect that the proposed data utilization framework will provide valuable insights for efficiently addressing continuous-time output feedback LQR problems. Furthermore, assertion (b) of Theorem \ref{thm3} may facilitate extending the state feedback data-enabled policy optimization algorithm in \cite{Deepo} to the output feedback scenario.

%%%%%%%%%%%%%%%%%%%%%%%%%%%%%%%%%%%%%%%%%
\section{Numerical Experiments}\label{section6}
This section presents numerical experiments to validate the feasibility and computational advantages of Algorithm \ref{al1} and Algorithm \ref{al2}. In all examples, the PE input $u_t=[u_t^1,\cdots,u_t^m]^\top\in\mathbb{R}^m$ is constructed such that each component is independently generated as $u_t^{k}=\sum_{i=1}^{100}c_{k,i}\sin(a_{k,i}t+b_{k,i})$, $k=1,\cdots,m$, where $a_{k,i}$, $b_{k,i}$ are sampled uniformly at random from $(0,2\pi)$, and $c_{k,i}$ from $(0,1)$.

\begin{example}[SO Problems]\label{ex2}
	First, consider the third-order single-input single-output (SISO) aircraft system in \cite{Deng-DT-output}, with system matrices given by
	$$A=\begin{bmatrix}
0.906488&0.0816012&-0.0005\\0.0741349&0.90121&-0.0007083\\0&0&0.132655
	\end{bmatrix},$$
    $$B=\begin{bmatrix}
		-0.00150808&-0.0096&0.867345
	\end{bmatrix}^\top,$$
	$$C=\begin{bmatrix}
		1&0&0
	\end{bmatrix},\quad Q=100,\quad R=1.$$
The initial state $x_0$, the eigenvalues of $A-LC$, $A_{\epsilon}$, and the initial state $\eta_{0}^{\epsilon}$ of system (\ref{A_e}) are chosen identical to those in \cite{Deng-DT-output}. We initialize the random number generator in MATLAB using the command rng(3). The initial stabilizing controller for Algorithm \ref{al1} is set as $K^0=\mathbf{0}$, and the initial matrix parameter for Algorithm \ref{al2} is set as $P^0=10^5I_9$; we set $\epsilon=1$ for all algorithms. The optimal output feedback gain obtained via Algorithm \ref{al1} is
	$$\begin{aligned}
		K_{PI}^*=[-0.1029,-0.2900,0.0377,-30.0794,264.5387,-286.4849,-37.1708,39.9982,32.9480].
	\end{aligned}$$
    It is very close to the true optimal output feedback gain reported in \cite{Deng-DT-output}.
	Compared with the LS-based PI algorithm in \cite{Deng-DT-output}, which requires $5$ iterations and has a runtime of $0.017086$ seconds, Algorithm \ref{al1} only needs $4$ iterations to reduce $\Vert K^{i+1}-K^{i}\Vert$ to $9.2385\times 10^{-5}$, with a runtime of $0.004024$ seconds. The optimal output feedback gain obtained via Algorithm \ref{al2} is
	$$\begin{aligned}
		K_{VI}^*=[-0.1029,-0.2900,0.0377,-30.0765,264.5137,-286.4590,-37.1673,39.9946,32.9448].
	\end{aligned}$$
	Compared with the LS-based VI algorithm in \cite{Deng-DT-output}, which requires $217$ iterations and has a runtime of $0.363406$ seconds, Algorithm \ref{al2} only needs $185$ iterations to reduce $\Vert P^{i+1}-P^{i}\Vert$ to $0.95764$, with a runtime of $0.033872$ seconds.

	Next, we randomly generate 50 third-order SISO-LTI systems with $Q=2$ and $R=1$. The eigenvalues of $A-LC$ are chosen as $-0.7$, $0.6$ and $0.8$; $A_{\epsilon}=\mathrm{diag}(-0.7,0.6,0.8)$. $x_0$ and $\eta_{0}^{\epsilon}$ are sampled uniformly from $[-0.5,0.5]$. We set $K^0=\mathbf{0}$ for PI algorithms, $P^0=10^3I_9$ for VI algorithms, and $\epsilon=0.01$. 
	Note that LS-based algorithms in \cite{Deng-DT-output} require regression matrices to have full row rank (the full row rank property of the PI regression matrix follows from that of the VI regression matrix), while Algorithm \ref{al1} and Algorithm \ref{al2} focus on the full row rank property of $V_0$. Therefore, we first compare the minimum singular value of the regression matrix from the LS-based VI algorithm with that of $V_0$, as shown in Figure \ref{f3}.	
	All orders of magnitude for the LS-based VI regression matrices are less than or equal to $10^{-7}$, whereas only $1$ value associated with $V_0$ has an order of magnitude below $10^{-7}$. 
    This indicates that regression matrices may be closer to singularity, leading to slow convergence or even non-convergence of the corresponding algorithms. In contrast, the proposed algorithms alleviate this problem to some extent. The average runtime and average iteration counts are reported in Table \ref{ta1} and Table \ref{ta2}. The above experimental results show that Algorithm \ref{al1} and Algorithm \ref{al2} are indeed more efficient.
    \begin{figure}[H] 		
    \centerline{\includegraphics[width=0.6\textwidth]{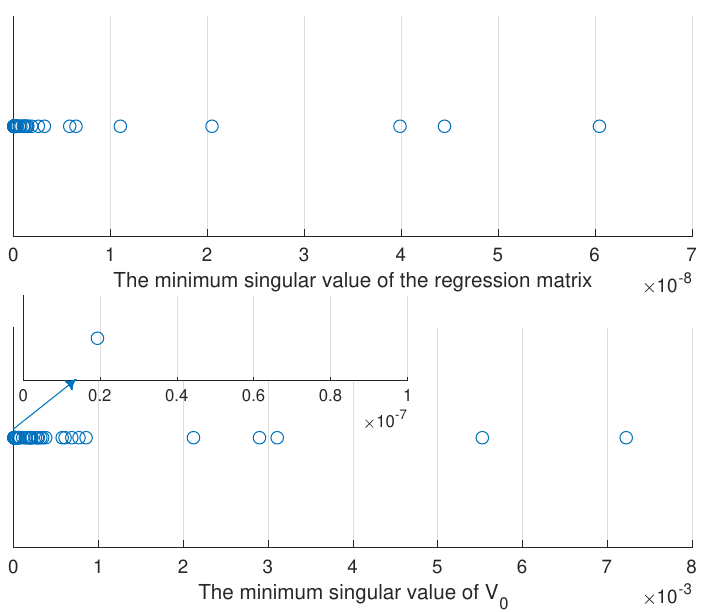}} 		
    \caption{Comparison of minimum singular value.} 		
    \label{f3} 	
    \end{figure}

	\begin{table}[H]
		\centering
		\begin{tabular}{|c|c|c|}
			\hline
			 & LS-based PI & Algorithm \ref{al1} \\ \hline
			Average runtime (s) &  0.0072 & 0.0018\\\hline
			Average iteration count & 3.24 & 3.24\\\hline
		\end{tabular}
		\caption{Efficiency comparison of PI Algorithms.}
		\label{ta1}
	\end{table}

	\begin{table}[H]
		\centering
		\begin{tabular}{|c|c|c|}
			\hline
			 & LS-based VI & Algorithm \ref{al2} \\ \hline
			Average runtime (s) & 0.3180& 0.0070\\\hline
			Average iteration count & 83.82 & 31.42 \\\hline
		\end{tabular}
		\caption{Efficiency comparison of VI Algorithms.}
		\label{ta2}
	\end{table}
\end{example}

\begin{example}[MO Problems]\label{ex3}
Consider the fourth-order coupled RLC circuit (two inductors, two capacitors, and three resistors, all component values equal to 1) from Example 2.5-2 in \cite{circuit}. Discretization at sampling period $T_s=0.1s$ yields the following system matrices.
	\begin{align}\label{con}
			\nonumber&A=\begin{bmatrix}
		0.90031&-0.00015&0.09048&-0.00452\\
       -0.00015&0.90031&0.00452&-0.09048\\
       -0.09048&-0.00452& 0.90483& -0.09033\\
        0.00452&0.09048& -0.09033& 0.90483
			\end{bmatrix},\\\nonumber
		&B=\begin{bmatrix}
	0.004678&0.00015&0.09516&-0.00467\\
       -0.00015&-0.00468&-0.00467&0.09516
		\end{bmatrix}^\top,\\
	&C=\begin{bmatrix}
		1&1&0&0\\
		0&1&0&0
	\end{bmatrix}, Q=\begin{bmatrix}
	    0.1&-0.1\\-0.1&0.3
	\end{bmatrix}, R=0.1I_2.
	\end{align}    
	The eigenvalues of $A-LC$ are chosen as $0.8994$, $-0.6$, $0.7$ and $0$; $A_{\epsilon}=\mathrm{diag}(0.8994,-0.6,0.7,0)$. We initialize the random number generator in MATLAB using the command rng(1). $x_0$ and $\eta_{0}^{\epsilon}$ follow zero-mean Gaussian distributions with covariance matrix $I_4$. The rank of the LS-based VI regression matrix in \cite{Deng-DT-output} is $152$, which is smaller than its number of rows, $253$; hence the corresponding LS-based algorithms are infeasible. In contrast, for Algorithm \ref{al1} and Algorithm \ref{al2}, the projected $V_0$ is guaranteed to have full row rank. We set $\epsilon=0.001$. For Algorithm \ref{al1}, we set $K^0=\mathbf{0}$; its convergence is shown in Figure \ref{f4}. Algorithm \ref{al1} reduces $\Vert K^{i+1}-K^i\Vert$ to $2.2513\times 10^{-5}$ after $4$ iterations. Noting that the true optimal cost $x_0^\top P^{K_x^*}x_0=2.2449$ and the calculated optimal cost $v_0^\top[I_{n_v}, K_{PI}^{*\top}]\Theta^{K_{PI}^{*}}[I_{n_v},K_{PI}^{*\top}]^\top v_0=2.2449$ differ by $5.3155\times 10^{-9}$, $K_{PI}^*$ is thus a near-optimal solution.
	For Algorithm \ref{al2}, we set $P^0=10^3I_{18}$; its convergence is shown in Figure \ref{f6}. 
    Algorithm \ref{al2} reduces $\Vert P^{i+1}-P^i\Vert$ to $8.2496\times 10^{-5}$ after $76$ iterations.
    Noting that the calculated optimal cost $v_0^\top[I_{n_v}, K_{VI}^{*\top}]\Theta^{K_{VI}^{*}}[I_{n_v},K_{VI}^{*\top}]^\top v_0=2.2451$ differs from $x_0^\top P^{K_x^*}x_0$ by $1.7453\times 10^{-4}$, $K_{VI}^*$ is thus a near-optimal solution.

    \begin{figure}[h]
    \centering
    \begin{minipage}{0.48\textwidth}
        \centering
        \includegraphics[width=\textwidth]{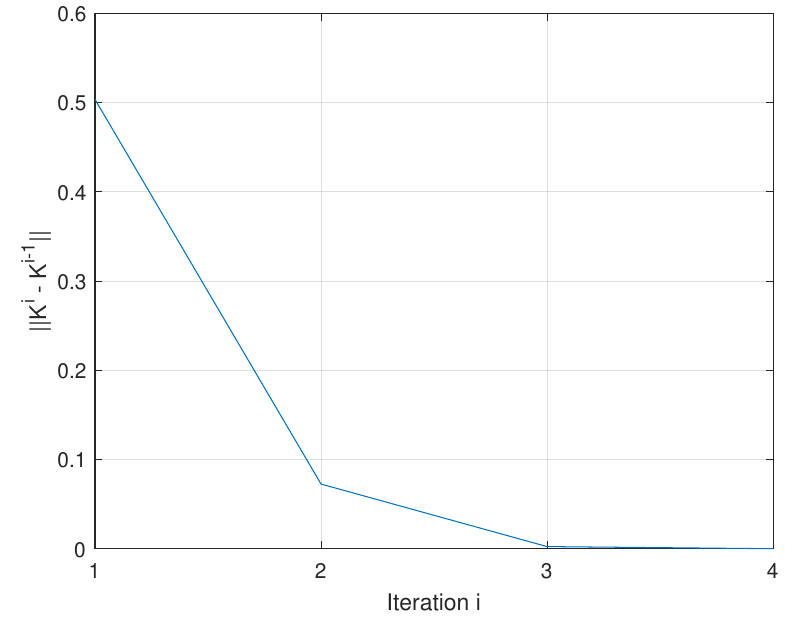}
        
        Fig. \refstepcounter{figure}\thefigure: Convergence of Algorithm \ref{al1}.
        \label{f4}
    \end{minipage}
    \hfill
    \begin{minipage}{0.48\textwidth}
        \centering
        \includegraphics[width=\textwidth]{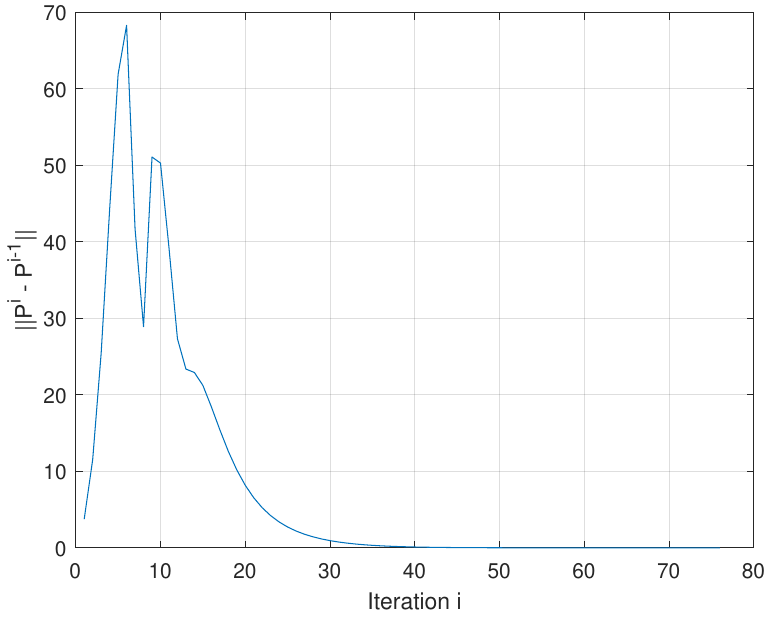}
        
        Fig. \refstepcounter{figure}\thefigure: Convergence of Algorithm \ref{al2}.
        \label{f6}
    \end{minipage}
\end{figure}

For comparison, the zeroth‑order policy gradient algorithm in \cite{Sof} for model-free optimal SOF control (rollout length $50$, $128$ parallel rollouts, $128$ perturbations, learning rate $0.1$, initial gain $K_y^0=
\left[\begin{smallmatrix}
    0&-1\\0&-2
\end{smallmatrix}\right]$) yields a calculated optimal cost of $2.2361$ after $300$ iterations, while the true expected optimal cost is $\mathrm{trace}(P^{K_x^*})=1.9072$. This algorithm requires large amounts of data due to parallel rollouts and is sensitive to the learning rate and initial gain. In contrast, our algorithms achieve optimality equivalent to the state feedback controller, and are insensitive to hyperparameters and significantly more data-efficient.

	Consider the uncontrollable but stabilizable LTI system with
	$$A=\begin{bmatrix}
0.3706&0.1537&0&0\\0.5123&0.3739&0&0\\0&0&0.5443&0\\0&0&0&0.7685
	\end{bmatrix},$$
    $$B=\begin{bmatrix}
    0.1174&0.8643&0.3159&0\\
    0.5487&0.8189&0.9594&0
	\end{bmatrix}^\top,C=\begin{bmatrix}
		1&1&1&0\\
		0&1&0&1
	\end{bmatrix},$$
$Q=R=I_2$.	
The remaining settings are the same as those in the controllable MO example.  
This setup ensures that $F$ and $V_0$ have full row rank. Algorithm \ref{al1} reduces 
$\Vert K^{i+1}-K^{i}\Vert$ to $4.9274\times 10^{-4}$ after $3$ iterations. Noting that the true optimal cost $x_0^\top P^{K_x^*}x_0=0.9318$ and the calculated optimal cost $v_0^\top[I_{n_v}, K_{PI}^{*\top}]\Theta^{K_{PI}^{*}}[I_{n_v},K_{PI}^{*\top}]^\top v_0=0.9318$ differ by $1.4503\times 10^{-11}$, $K_{PI}^*$ is thus a near-optimal solution.
	Algorithm \ref{al2} reduces $\Vert P^{i+1}-P^{i}\Vert$ to $8.2238\times 10^{-5}$ after $70$ iterations. Noting that the calculated optimal cost $v_0^\top[I_{n_v}, K_{VI}^{*\top}]\Theta^{K_{VI}^{*}}[I_{n_v},K_{VI}^{*\top}]^\top v_0=0.9318$ differs from $x_0^\top P^{K_x^*}x_0$ by $3.1623\times 10^{-4}$, $K_{VI}^*$ is thus a near-optimal solution. 
\end{example}

\begin{example}[Robustness Problem]
Consider the same system (\ref{con}) with identical user-defined parameters. Let the absolute upper bounds of noises
	$\{w_t\}$ and $\{e_t\}$ both be $W_{\max}$. When $W_{\max}=10^{-4}$, $100$ trials are performed to construct the regression matrices for the LS-based VI method in \cite{Deng-DT-output}. The average row-rank is $221.58$, with none achieving full row rank; hence, LS-based algorithms remain infeasible. 
    We set $W_{\max}$ to $10^{-3}$, $10^{-4}$ and $10^{-6}$. 
    Define $\Delta_{xPx}$ as the value obtained at the final iteration of $\left\vert x_0^\top P^{K_x^*}x_0-v_0^\top[I_{n_v}, K^{i+1\top}]\Theta^{i+1}[I_{n_v},K^{i+1\top}]^\top v_0\right\vert$. The iteration results are shown in Table \ref{ta3} and Table \ref{ta4}.
	Therefore, when $W_{\max}$ is small, the results of Algorithm \ref{al1} and Algorithm \ref{al2} are very close to those obtained under exact data.

\begin{table}[h]
	\centering
	\begin{tabular}{|c|c|c|c|}
		\hline
		$W_{\max}$ &$10^{-3}$  &  $10^{-4}$ & $10^{-6}$\\ \hline
       Iterations &3 &4& 4\\\hline
		$\Delta_{xPx}$ & 0.2065 &  0.0950 &$3.0353\times 10^{-4}$\\\hline
	\end{tabular}
	\caption{Robustness of Algorithm \ref{al1} to different $W_{\max}$.}
	\label{ta3}
\end{table}

\begin{table}[h]
	\centering  
	\begin{tabular}{|c|c|c|c|}
		\hline
		$W_{\max}$ &$10^{-3}$  &  $10^{-4}$ & $10^{-6}$\\ \hline
        Iterations& 81 & 85 & 62 \\\hline
		$\Delta_{xPx}$  &  0.2116 & 0.0990 & 0.0040\\\hline
	\end{tabular}
	\caption{Robustness of Algorithm \ref{al2} to different $W_{\max}$.}
	\label{ta4}
\end{table}
    
\end{example}

\section{Conclusion}\label{section7}	
This work develops a general and efficient data-driven framework for solving the output feedback LQR problem in a model-free manner. Specifically, we focus on the data properties of the state parameterization method and provide theoretical foundations for data-parameterized output feedback controllers. The proposed algorithms fully exploit the data structure and offer clear advantages in data usage conditions and data volume requirements. This framework is easily extendable to various linear optimal control problems, and the data parameterization theory can be tightly integrated with multiple data-driven approaches. However, state parameterization is essentially a redundant representation along with increased computational complexity; thus, achieving model-free output feedback optimal control in high-dimensional scenarios remains a key research topic for future exploration. Moreover, extending the proposed algorithms to more general noise scenarios or linear quadratic Gaussian problems is another core direction for our subsequent work.

\section*{References}
\bibliographystyle{IEEEtran}
\bibliography{ref}

\end{document}